\documentclass[11pt,reqno]{amsart}
\usepackage{xcolor}
\usepackage[colorlinks=true,allcolors=blue,backref=page]{hyperref}
\usepackage{amsmath, amssymb, amsthm}

\usepackage{mathrsfs}
\usepackage{mathtools}
\usepackage[noabbrev,capitalize,nameinlink]{cleveref}
\usepackage{aliascnt}
\usepackage{fullpage}
\usepackage[noadjust]{cite}
\usepackage{graphics}
\usepackage{pifont}
\usepackage{tikz}
\usepackage{tikz-cd}
\usepackage{bbm}
\usepackage[T1]{fontenc}
\usepackage{etoolbox}

\usetikzlibrary{arrows.meta}

\usepackage{environ}
\usepackage{framed}
\usepackage{url}
\usepackage[linesnumbered,ruled,vlined]{algorithm2e}
\usepackage[noend]{algpseudocode}
\usepackage[labelfont=bf]{caption}
\usepackage[framemethod=tikz]{mdframed}
\usepackage{appendix}
\usepackage{graphicx}
\usepackage[textsize=tiny]{todonotes}
\usepackage{tcolorbox}
\usepackage{enumerate}
\usepackage[shortlabels]{enumitem}
\usepackage{physics}
\allowdisplaybreaks[1]

\apptocmd{\sloppy}{\hbadness 10000\relax}{}{} 

\crefname{equation}{}{}
\crefname{algocf}{Algorithm}{Algorithms}
\crefname{conjecture}{Conjecture}{Conjectures} 
\AtBeginEnvironment{appendices}{\crefalias{section}{appendix}} 

\crefformat{enumi}{#2#1#3}
\crefrangeformat{enumi}{#3#1#4 to~#5#2#6}
\crefmultiformat{enumi}{#2#1#3}%
{ and~#2#1#3}{, #2#1#3}{ and~#2#1#3}

\numberwithin{equation}{section}
\newtheorem{theorem}{Theorem}[section]

\newaliascnt{proposition}{theorem}
\newtheorem{proposition}[proposition]{Proposition}
\aliascntresetthe{proposition}
\crefname{proposition}{Proposition}{Propositions}

\newaliascnt{lemma}{theorem}
\newtheorem{lemma}[lemma]{Lemma}
\aliascntresetthe{lemma}
\crefname{lemma}{Lemma}{Lemmas}

\newaliascnt{corollary}{theorem}
\newtheorem{corollary}[corollary]{Corollary}
\aliascntresetthe{corollary}
\crefname{corollary}{Corollary}{Corollaries}

\crefname{subsubsection}{Step}{Steps}

\crefname{claim}{Claim}{Claims}

\newtheorem*{question*}{Question}

\theoremstyle{definition}

\newtheorem*{definition*}{Definition}

\theoremstyle{remark}

\newcommand{\snorm}[1]{\lVert#1\rVert}

\newcommand{\one}{\mathbbm{1}}

\newcommand{\mc}{\mathcal}

\newcommand{\wt}{\widetilde}

\renewcommand{\le}{\leqslant}
\renewcommand{\ge}{\geqslant}

\newcommand\Z{\mathbb{Z}}

\newcommand\R{\mathbb{R}}

\newcommand\PP{\mathbb{P}}
\newcommand\E{\mathbb{E}}

\usepackage{bm}

\newcommand{\Id}{\mathrm{Id}}

\title{On the maxima of Littlewood polynomials on $[-1,1]$}

\author[Letwin]{Brayden Letwin}
\address{Department of Mathematics, University of Washington, Seattle, WA 98195}
\email{letwin@uw.edu}

\author[Sawhney]{Mehtaab Sawhney}
\address{Department of Mathematics, Columbia University, New York, NY 10027 and OpenAI}
\email{\{m.sawhney@columbia.edu, msawhney@openai.com\}}

\begin{document}

\begin{abstract}
A Littlewood polynomial is a polynomial of the form
\[
    f_n(x)=\sum_{k=0}^n \varepsilon_k x^k
\]
with $\varepsilon_k\in\{-1, 1\}$. Let $(\varepsilon_k)_{k \ge 0}$ be i.i.d. Rademacher coefficients. We show that the lower envelope of $\max_{x\in[-1,1]}|f_n(x)|$ is determined by the small-ball probability of a certain Gaussian process. In particular, almost surely,
\[
\liminf_{n\to\infty} \frac{\log(\max_{x\in[-1,1]}|f_n(x)|/\sqrt n)}{(\log\log n)^{1/3}}
    = -\Big(\frac{3\pi^2}{4}\Big)^{1/3}.
\]
\end{abstract}

\maketitle

\section{Introduction} \label{sec:1}

\subsection{Introduction to Littlewood polynomials} \label{subsec:1.1}
A Littlewood polynomial of degree $n$ is a polynomial of the form
\[
    f_n(x)=\sum_{k=0}^n \varepsilon_k x^k
\]
with $\varepsilon_k\in\{-1,1\}$. We study random Littlewood polynomials with i.i.d.\ Rademacher coefficients on the interval $[-1,1]$.

Throughout the paper we write
\[
    \snorm{f_n}_\infty=\max_{x \in [-1, 1]}|f_n(x)|.
\]
As part of their seminal work on random polynomials, Salem and Zygmund \cite[Theorem~(6.1.1)]{SZ54} proved, as a consequence of the law of the iterated logarithm, that almost surely
\begin{equation} \label{eq:1.1}
\limsup_{n \to \infty} \frac{\snorm{f_n}_\infty}{\sqrt{n\log\log n}} = \sqrt{2}.
\end{equation}

A natural question, raised in \cite{SZ54} and later reiterated by Erd\H{o}s \cite{Erd61}, is to determine the corresponding lower envelope. Our first main result relates this lower envelope to a Gaussian process.
\begin{theorem} \label{thm:1.1}
Let $B$ be a standard Brownian motion and, for $\delta>0$, define
\[
F(\delta)=\PP\bigg(\sup_{t\ge0}\bigg|\int_0^1 e^{-st}\,dB_s\bigg|\le\delta\bigg).
\]
Then $F$ is continuous and strictly increasing on $(0,\infty)$, and hence admits an inverse $F^{-1}\colon(0,1)\to(0,\infty)$. Furthermore, almost surely
\[
\liminf_{n\to\infty} \frac{\snorm{f_n}_\infty}{\sqrt{n}\,F^{-1}(\log^{-1/2}n)} = 1.
\]
\end{theorem}

Given \cref{thm:1.1}, it is natural to seek a sharper understanding of $F(\delta)$. The Gaussian process underlying \cref{thm:1.1} was previously studied by Gao, Li, and Wellner \cite{GLW10}, who proved that
\[
    \log F(\delta) \asymp -\log^3(1/\delta).
\]
Our second main result identifies the leading constant.
\begin{theorem} \label{thm:1.2}
For $\delta\in(0,1/4)$,
\[
    \log F(\delta) = -\frac{2}{3\pi^2}\log^3(1/\delta) + o\big(\log^3(1/\delta)\big).
\]
\end{theorem}

In fact, we prove a quantitative version of \cref{thm:1.2}; see \cref{thm:2.7}. Combining \cref{thm:1.2} and \cref{thm:1.1} immediately yields the asymptotic stated in the abstract. We also note that the method underlying \cref{thm:1.2} is fairly robust: the key probabilistic estimates sandwich the relevant event between two $L^2$ events that can be handled by spectral methods. In particular, the same strategy should apply to other Gaussian processes arising from sufficiently smooth kernels, such as those appearing in work of Aurzada, Gao, K\"{u}hn, Li, and Shao \cite{AGKLS11}; we do not pursue this direction.

\subsection{Notation and conventions}\label{subsec:1.2}
We use the asymptotic notations $\gg$, $\ll$, $\Omega(\cdot)$, $O(\cdot)$, and $o(\cdot)$ in the standard way. Thus $X=O(Y)$ and $X\ll Y$ both mean that $|X|\le CY$ for some absolute constant $C$ (which may change from line to line), while $X=\Omega(Y)$ and $X\gg Y$ mean that $Y\ll X$. We write $X=o(Y)$ if $X/Y\to0$ in the relevant limit. If $X$ and $Y$ are positive quantities, then $X\asymp Y$ means $Y\ll X\ll Y$. All logarithms are natural unless explicitly indicated otherwise. We also record our Fourier-transform convention: for $f\in L^1(\R)$,
\[
    \widehat f(\xi)=\int_{\R} f(x)e^{-i\xi x}\,dx,
\]
so that the inversion formula reads
\[
    f(x)=\frac{1}{2\pi}\int_{\R}\widehat f(\xi)e^{i\xi x}\,d\xi.
\]

\subsection{Overview of the proof} \label{subsec:1.3}
We begin with a brief sketch of the proof of \cref{thm:1.1}. For $x\in(0,1]$, we write $x=e^{-t/n}$ with $t\ge0$, and for $x\in[-1,0)$, we write $x=-e^{-t/n}$ with $t\ge0$. This introduces logarithmic coordinates near the two endpoints, and an exact identity expresses $\snorm{f_n}_{\infty}$ in terms of the suprema of the resulting profiles together with the value at $x=0$. To be precise, we set
\begin{equation} \label{eq:1.2}
    f_{n,t}^{+}=n^{-1/2}f_n(e^{-t/n}),
    \qquad
    f_{n,t}^{-}=n^{-1/2}f_n(-e^{-t/n}).
\end{equation}
We also define
\[
    S_s^n=\frac{1}{\sqrt n}\sum_{0\le k\le \lfloor ns\rfloor}\varepsilon_k.
\]
Via the Koml\'os--Major--Tusn\'ady (KMT) approximation \cite{KMT75,KMT76}, we may couple $S_s^n$ with a standard Brownian motion $B_s$ so that $|S_s^n-B_s|\ll (\log n)/\sqrt n$. By coupling the even and odd coefficients separately, one is naturally led to stochastic integrals of the form
\[
    \int_0^1 e^{-st}\,dB_s^{+},
    \qquad
    \int_0^1 e^{-st}\,dB_s^{-},
\]
where $B_s^{+}$ and $B_s^{-}$ are independent Brownian motions. This reduces the problem to the study of independent copies of the Gaussian process $(Y_t)_{t\ge0}$ defined by
\[
    Y_t=\int_0^1 e^{-st}\,dB_s,
\]
and of the associated small-ball probability
\[
    F(\delta)=\PP\big(\sup_{t\ge0}|Y_t|\le\delta\big),
\]
whose asymptotics as $\delta\downarrow0$ drive the rest of the argument.

Now we discuss the proof of \cref{thm:1.2}. Setting $Z_t=e^{t/2}Y_{e^t}$, we find that
\[
    \E[Z_tZ_s]=\frac12 \operatorname{sech}\!\Big(\frac{s-t}{2}\Big)(1-e^{-e^s-e^t}).
\]
Using a comparison argument, it is enough to study the stationary Gaussian process $X_t$ with covariance
\[
    \E[X_tX_s]=\frac12 \operatorname{sech}\!\Big(\frac{s-t}{2}\Big)
\]
and to bound
\[
    G(\delta)=\PP\big(\sup_{t\ge0}e^{-t/2}|X_t|\le\delta\big).
\]
The next step is to compare this $L^\infty$ event with the $L^2$ quantity
\[
    H(\delta)=\PP\Big(\int_0^{\infty} e^{-t}X_t^2\,dt\le\delta^2\Big).
\]
Since
\[
    \E\Big[\int_{4\log(1/\delta)}^{\infty} e^{-t}|X_t|^2\,dt\Big]\le \delta^4/2,
\]
Markov's inequality gives
\[
    \PP\Big(\int_{4\log(1/\delta)}^{\infty} e^{-t}|X_t|^2\,dt\le 4\delta^2\log(1/\delta)\Big)\ge \frac12
 \]
for all sufficiently small $\delta$, and on the event defining $G(\delta)$ we have
\[
    \int_0^{4\log(1/\delta)} e^{-t}X_t^2\,dt \le 4\delta^2\log(1/\delta),
\]
so the Gaussian correlation inequality \cite[Theorem~1]{Royen14} gives
\[
    G(\delta)\le 2H(4\delta\log(1/\delta)).
\]
For the reverse implication, after introducing a smooth cutoff $w$ we prove the local comparison bound
\[
    \PP\Big(\max_{-1\le t\le1}|w(t)X_t| \ge L\Big(\int_{-1}^1 |w(t)X_t|^2\,dt\Big)^{1/2}\Big)
    \ll e^{-\Omega(L/\log^2(L+e))}.
\]
This yields
\[
    H(\delta)\exp\big(-O(\log^2(1/\delta))\big)\le G\big(C\delta\log^4(1/\delta)\big)
\]
and since
\[
    C\theta\log^4(1/\theta)\le \delta,
    \qquad
    \theta=\frac{\delta}{4C\log^4(1/\delta)},
\]
for all sufficiently small $\delta$, monotonicity gives
\[
    H\Big(\frac{\delta}{4C\log^4(1/\delta)}\Big)\exp\big(-O(\log^2(1/\delta))\big)\ll G(\delta).
\]
The proof of this comparison ultimately rests on Fourier analysis after introducing the smooth cutoff weight, and the appendix then establishes the corresponding sharp $L^2$ small-ball asymptotic, with error term $O(\log(1/\delta)^{5/2}\sqrt{\log\log(1/\delta)})$.

Returning to \cref{thm:1.1}, the KMT coupling transfers the Gaussian estimates back to Littlewood polynomials. The argument then proceeds in the spirit of the law of the iterated logarithm: a sparse mesh captures the relevant scales, the lower bound follows from a first-moment argument, and the upper bound comes from splitting the polynomial into an old part and an independent fresh block. A final subtlety is that the asymptotic for $F(\delta)$ is not quite precise enough to control the difference between $F(\delta)$ and $F((1+\eta)\delta)$ directly. To handle this, we use the Gaussian $B$-inequality of Cordero-Erausquin, Fradelizi, and Maurey \cite{CFM04}, which implies that $t\mapsto F(e^t)$ is log-concave and yields the continuity estimates that we need.

\subsection{Acknowledgments} \label{subsec:1.4}
BL thanks Dan Mikulincer for helpful discussions. This research was conducted while MS held a Clay Research Fellowship.

\subsection*{AI Disclosure}
ChatGPT Pro drew our attention to the Gaussian $B$-inequality used in \cref{prop:4.1} (see \href{https://chatgpt.com/share/69d1e613-3604-83e8-9b95-32a47dd030f0}{link}). In addition, ChatGPT Codex was used to help write the manuscript.
\section{\texorpdfstring{$L^{\infty}$}{L-infinity} small-ball probabilities for a Gaussian process} \label{sec:2}
\subsection{Preparatory steps}\label{subsec:2.1}
Before proving \cref{thm:1.2}, we make a number of convenient reductions. Let $B$ be a standard Brownian motion on $[0,\infty)$, and for $t > 0$ define
\begin{equation} \label{eq:2.1}
    Y_t=\int_0^1 e^{-ut}\,dB_u,
    \qquad
    \wt Y_t=\int_0^{\infty} e^{-ut}\,dB_u.
\end{equation}
By It\^o's isometry,
\[
    \E[Y_sY_t]=\int_0^1 e^{-u(s+t)}\,du = \frac{1-e^{-s-t}}{s+t},
    \qquad
    \E[\wt Y_s\wt Y_t]=\frac{1}{s+t}.
\]
Moreover, $\wt Y_s$ has the same distribution as $Y_s+e^{-s}\wt Y_s'$, where $(\wt Y_t')_{t>0}$ is an independent copy of $\wt Y$. We also define
\begin{equation} \label{eq:2.2}
    Z_t=e^{t/2}Y_{e^t},
    \qquad
    X_t=e^{t/2}\wt Y_{e^t}.
\end{equation}
If $K(t)=\operatorname{sech}(t/2)/2$, then
\begin{equation} \label{eq:2.3}
    \E[Z_sZ_t]=K(s-t)(1-e^{-e^s-e^t}),
    \qquad
    \E[X_sX_t]=K(s-t).
\end{equation}
Note also that $X_t$ has the same distribution as $Z_t+e^{-e^t}X_t'$, where $(X_t')_{t\in\R}$ is an independent copy of $X$. Since $\E[X_tX_s]$ depends only on $t-s$, the process $X$ is stationary.

Recall that
\[
    F(\delta)=\PP\Big(\sup_{t\ge0}|Y_t|\le\delta\Big)=\PP\Big(\sup_{t \in \R} e^{-t/2}|Z_t|\le\delta\Big),
\]
and define
\[
    G(\delta)=\PP\Big(\sup_{t\ge0} e^{-t/2}|X_t|\le\delta\Big).
\]
All applications of the Gaussian correlation inequality, the Gaussian $B$-inequality, and Anderson's inequality below are obtained by first applying the corresponding finite-dimensional statement to process values on a finite mesh (or to Riemann sums for the $L^2$ events) and then passing to the limit using continuity and monotone convergence.
Our first step is to show that $F$ and $G$ are closely related.

\begin{lemma}\label{lem:2.1}
Let $\delta\in(0,1/4)$. Then
\[
    \exp(-O(\log(1/\delta)^2))G(\delta)\le F(\delta)
\]
and
\[
    F(\delta/2)\,G(1)\,\exp\big(-O(\log^2(1/\delta)\log\log(1/\delta))\big)\ll G(\delta).
\]
\end{lemma}
\begin{proof}
We first compare $F$ and $G$ on the range $t\ge1$. By Royen's Gaussian correlation inequality,
\[
    F(\delta)
    \ge \PP\Big(\sup_{t\ge1}|Y_t|\le\delta\Big)
       \PP\Big(\sup_{0\le t\le1}|Y_t|\le\delta\Big)
\]
and since
\[
    X_t \overset{d}= Z_t+e^{-e^t}X_t'
\]
with $X_t'$ independent of $Z_t$, Anderson's inequality \cite[Theorem~1]{Anderson55} gives
\[
    \PP\Big(\sup_{t\ge0} e^{-t/2}|Z_t|\le\delta\Big)\ge G(\delta).
\]
Moreover,
\[
    \PP\Big(\sup_{t\ge1}|Y_t|\le\delta\Big)
    =\PP\Big(\sup_{t\ge0} e^{-t/2}|Z_t|\le\delta\Big),
\]
so it remains only to control the contribution from the short interval $[0,1]$, namely to prove that
\[
    \PP\Big(\sup_{0\le t\le1}|Y_t|\le\delta\Big)\ge \exp(-O(\log(1/\delta)^2)).
\]
On this interval the process is analytic in $t$, so we may control it through its Taylor coefficients. The point is to split the series into a finite-dimensional part, where we ask that each coefficient be individually small, and a tail part, which we show is typically negligible.
For $0\le t\le1$ we have the convergent expansion
\[
    Y_t=\sum_{m=0}^{\infty}(-t)^m a_m,
    \qquad
    a_m=\int_0^1 \frac{u^m}{m!}\,dB_u,
\]
and therefore
\[
    \sup_{0\le t\le1}|Y_t|\le \sum_{m=0}^{\infty}|a_m|.
\]
Moreover, each $a_m$ is a centered Gaussian with
\[
    \E[a_m^2]=\int_0^1 \frac{u^{2m}}{(m!)^2}\,du
    =\frac{1}{(2m+1)(m!)^2}.
\]
Set \(M=\lceil 10\log(1/\delta)\rceil\) and \(\lambda=\delta/(4(M+1))\), and write
\[
    A_M=\bigcap_{m=0}^{M}\{|a_m|\le \lambda\},
    \qquad
    B_M=\Big\{\sum_{m>M}|a_m|\le \delta/2\Big\}.
\]
The event $A_M$ says that each of the first $M+1$ coefficients is at most $\lambda$ in absolute value, while $B_M$ controls the remaining tail. Since $(M+1)\lambda=\delta/4$, on $A_M$ we have
\[
    \sum_{m=0}^{M}|a_m|\le (M+1)\lambda=\delta/4,
\]
and therefore on $A_M\cap B_M$,
\[
    \sum_{m=0}^{\infty}|a_m|
    \le (M+1)\lambda+\delta/2
    \le \delta.
\]
Hence
\[
    \PP\Big(\sup_{0\le t\le1}|Y_t|\le\delta\Big)\ge \PP(A_M\cap B_M).
\]
We estimate the two pieces separately: $B_M$ is the tail event, and $A_M$ is a finite-dimensional small-ball event for the coefficient vector $(a_0,\dots,a_M)$.
Now, by Stirling,
\[
    \sum_{m>M}\big(\E[a_m^2]\big)^{1/2}\ll \sum_{m>M}\frac{1}{m!}\ll \frac{1}{(M+1)!},
\]
so in particular $\E[\sum_{m>M}|a_m|]\ll (M+1)!^{-1}\ll \delta^2$. Thus, by Markov's inequality,
\[
    \PP(B_M)\ge \frac12
\]
for all sufficiently small $\delta$. Furthermore, by \v{S}id\'ak's lemma,
\[
    \PP(A_M)\ge \prod_{m=0}^{M}\PP(|a_m|\le \lambda)\ge \exp(-O(\log(1/\delta)^2)). 
\]
Finally, applying Royen's inequality once more to the symmetric convex events $A_M$ and $B_M$, we obtain
\[
    \PP(A_M\cap B_M)\ge \PP(A_M)\PP(B_M)\ge \exp(-O(\log(1/\delta)^2)).
\]
This proves the first bound.

For the reverse bound, let
\[
    T=2\log\log(1/\delta).
\]
We may assume $\delta$ is sufficiently small. We again start from the decomposition
\[
    X_t \overset{d}= Z_t+e^{-e^t}X_t'
\]
with $X_t'$ independent of $Z_t$. This gives
\begin{align*}
G(\delta)
&= \PP\Big(\sup_{t\ge0}e^{-t/2}|X_t|\le\delta\Big) \\
&\ge \PP\Big(\sup_{t\ge0}e^{-t/2}|Z_t|\le\delta/2\Big)
   \PP\Big(\sup_{t\ge0}e^{-t/2-e^t}|X_t'|\le\delta/2\Big) \\
&\ge F(\delta/2)\,
   \PP\Big(\sup_{t\ge0}e^{-e^t}|X_t'|\le\delta/2\Big),
\end{align*}
where we used
\[
    \PP\Big(\sup_{t\ge0}e^{-t/2}|Z_t|\le\delta/2\Big)
    = \PP\Big(\sup_{s\ge1}|Y_s|\le\delta/2\Big)
    \ge F(\delta/2).
\]
To estimate the second factor, we split the event at time $T$. For $0\le t\le T$, the condition $|X_t'|\le \delta/2$ certainly implies
\[
    e^{-e^t}|X_t'|\le \delta/2.
\]
Also, for $t\ge T$ one has
\[
    e^{-e^t}\le \delta^2 e^{-t/2}
\]
for all sufficiently small $\delta$, and hence
$e^{-t/2}|X_t'|\le \delta^{-1}/2$ implies $e^{-e^t}|X_t'|\le \delta/2.$
Therefore, if
\[
    C_T=\Big\{\sup_{0\le t\le T}|X_t'|\le \delta/2\Big\}
    \qquad\text{and}\qquad
    D_T=\Big\{\sup_{t\ge T}e^{-t/2}|X_t'|\le \delta^{-1}/2\Big\},
\]
then $C_T\cap D_T$ is contained in the event
\[
    \Big\{\sup_{t\ge0}e^{-e^t}|X_t'|\le \delta/2\Big\}.
\]
Hence, by Royen's inequality,
\[
    \PP\Big(\sup_{t\ge0}e^{-e^t}|X_t'|\le \delta/2\Big)\ge \PP(C_T)\PP(D_T).
\]
By stationarity and another application of Royen's inequality,
\[
    \PP(C_T)\ge
    \Big(\PP\big(\sup_{0\le t\le1}|X_t|\le \delta/2\big)\Big)^{\lceil T\rceil}.
\]
Moreover, since $\delta^{-1}/2\ge1$ and restricting to $t\ge T$ only enlarges the event,
\[
    \PP(D_T)\ge G(1).
\]
Thus
\[
    G(\delta)\ge
    F(\delta/2)\,G(1)\,
    \Big(\PP\big(\sup_{0\le t\le1}|X_t|\le \delta/2\big)\Big)^{\lceil T\rceil}.
\]
This reduces the problem to bounding the small-ball probability of $X$ on a unit interval. By definition of $X_t$,
\[
    \PP\Big(\sup_{0\le t\le1}|X_t|\le\delta\Big)
    = \PP\Big(\sup_{1\le s\le e}s^{1/2}|\wt Y_s|\le\delta\Big)
    \ge \PP\Big(\sup_{1\le s\le e}|\wt Y_s|\le e^{-1/2}\delta\Big).
\]
We estimate this in the same spirit as before, now using the expansion of $\wt Y_t$ around $t=2$. As above, the idea is to control finitely many Taylor coefficients directly and then show that the remaining tail contributes very little.

Now write
\[
    \wt Y_t
    = \int_0^{\infty} e^{-2u}e^{-(t-2)u}\,dB_u
    = \sum_{m=0}^{\infty} (-(t-2))^m a_m,
\]
where
\[
    a_m=\int_0^{\infty}\frac{e^{-2u}u^m}{m!}\,dB_u.
\]
Since $|t-2|\le1$ for $t\in[1,e]$, we have
\[
    \sup_{1\le t\le e}|\wt Y_t|\le \sum_{m=0}^{\infty}|a_m|.
\]
Moreover,
\[
    \E[a_m^2]
    = \int_0^{\infty}\frac{e^{-4u}u^{2m}}{(m!)^2}\,du
    = \frac{(2m)!}{4^{2m+1}(m!)^2}
    \ll 4^{-m}.
\]
Let $M=\lceil C\log(1/\delta)\rceil$ and $\lambda=\frac{e^{-1/2}\delta}{4(M+1)},$ where $C>0$ is a sufficiently large absolute constant. If
\[
    A_M=\bigcap_{m=0}^{M}\{|a_m|\le \lambda\}
    \qquad\text{and}\qquad
    B_M=\Big\{\sum_{m>M}|a_m|\le e^{-1/2}\delta/2\Big\},
\]
then, just as before, $A_M$ controls the first $M+1$ coefficients and $B_M$ controls the tail. Consequently, on $A_M\cap B_M$ we have
\[
    \sup_{1\le t\le e}|\wt Y_t|
    \le (M+1)\lambda + e^{-1/2}\delta/2
    \le e^{-1/2}\delta.
\]
Hence
\[
    \PP\Big(\sup_{1\le t\le e}|\wt Y_t|\le e^{-1/2}\delta\Big)\ge \PP(A_M\cap B_M).
\]
If $\tau_m^2=\E[a_m^2]$, then $\tau_m\ll 2^{-m}$. Thus
\[
    \sum_{m>M}\tau_m\ll 2^{-M}\ll \delta
\]
for $C$ large enough, and so Markov's inequality gives
\[
    \PP(B_M)\ge \frac12.
\]
Thus the tail event $B_M$ has probability bounded below by an absolute constant. To handle the finite-dimensional event $A_M$, we use \v{S}id\'ak's lemma:
\[
    \PP(A_M)\ge \prod_{m=0}^{M}\PP(|a_m|\le \lambda)\ge \exp(-O(\log(1/\delta)^2)).
\]
A final application of Royen's inequality yields
\[
    \PP(A_M\cap B_M)\ge \PP(A_M)\PP(B_M)\ge \exp(-O(\log^2(1/\delta))),
\]
and hence
\[
    \PP\Big(\sup_{0\le t\le1}|X_t|\le\delta\Big)\ge \exp(-O(\log^2(1/\delta))).
\]
Substituting this into the previous estimate and recalling that $T\asymp \log\log(1/\delta)$ gives
\[
    G(\delta)\ge
    F(\delta/2)\,G(1)\,
    \exp\big(-O(\log^2(1/\delta)\log\log(1/\delta))\big),
\]
as required.
\end{proof}

We next construct a cutoff function that equals $1$ on a smaller interval and still has almost-exponential Fourier decay. This lets us localize the process $X$ without losing the Fourier decay needed to rule out cancellations. The construction is standard, but we record one for completeness.

\begin{lemma} \label{lem:2.2}
There exists a function $w\colon\R\to[0,1]$ such that $\operatorname{supp}(w)\subset[-1/2,1/2]$, $w(x)=1$ for $x\in[-1/4,1/4]$, and
\begin{equation} \label{eq:2.4}
    |\widehat w(\xi)| \ll e^{-\Omega(|\xi|/\log^2(|\xi|+e))}
\end{equation}
for all $\xi\in\R$.
\end{lemma}
\begin{proof}
Let $a_k=k\log^2 k$, and choose $C$ so that $\sum_{k\ge C} a_k^{-1}\le2^{-4}$. For $k\ge C$, let
\[
    f_k(t)=\frac{a_k}{2}\mathbbm{1}_{[-1/a_k,\,1/a_k]}(t),
\]
and write $g=\mathop{\ast}_{k\ge C} f_k$. Define
\[
    w(x)=\int_{x-3/8}^{x+3/8} g(t)\,dt.
\]
The support and plateau properties are immediate. Moreover,
\[
    \widehat w(\xi)=\frac{2\sin(3\xi/8)}{\xi}\prod_{k\ge C}\frac{\sin(\xi/a_k)}{\xi/a_k}.
\]
If $N=\max\{k\ge C:a_k\le|\xi|\}$, then $N\asymp |\xi|/\log^2(|\xi|+e)$, and for $k\le N/4$ we have $a_k\le|\xi|/2$ once $|\xi|$ is large. Since $|\sin(\xi/a_k)/(\xi/a_k)|\le1/2$ for such $k$,
\[
    |\widehat w(\xi)|\ll 2^{-N/4}\ll e^{-\Omega(|\xi|/\log^2(|\xi|+e))},
\]
which proves the claim.
\end{proof}

\subsection{\texorpdfstring{Passing to $L^2$ small-ball probabilities}{Passing to L2 small-ball probabilities}} \label{subsec:2.2}
The next step is to compare the $L^{\infty}$ event defining $G(\delta)$ with the $L^2$ event
\[
    H(\delta)=\PP\Big(\int_0^{\infty} e^{-t}X_t^2\,dt\le\delta^2\Big).
\]

The easy direction is a direct consequence of the Gaussian correlation inequality.
\begin{lemma}\label{lem:2.3}
There exists $\delta_0\in(0,1)$ such that whenever $0<\delta<\delta_0$,
\[
    H(4\delta\log(1/\delta))\ge \frac{G(\delta)}{2}.
\]
\end{lemma}
\begin{proof}
We truncate the weighted $L^2$ norm at time $4\log(1/\delta)$. On the event defining $G(\delta)$, the truncated part is automatically small, while the tail has expectation $\ll\delta^4$ and is therefore small with positive probability.

We may assume that $\delta$ is smaller than an absolute constant. If $\sup_{t\ge0} e^{-t/2}|X_t|\le\delta$, then
\[
    \int_0^{4\log(1/\delta)} e^{-t}|X_t|^2\,dt\le 4\delta^2\log(1/\delta).
\]
Also,
\[
    \E\Big[\int_{4\log(1/\delta)}^{\infty} e^{-t}|X_t|^2\,dt\Big]
    =\int_{4\log(1/\delta)}^{\infty} \frac{e^{-t}}{2}\,dt
    \le \delta^4/2.
\]
By Markov's inequality,
\[
    \PP\Big(\int_{4\log(1/\delta)}^{\infty} e^{-t}|X_t|^2\,dt\ge\delta^4\Big)\le\frac12.
\]
Therefore
\begin{align*}
\PP\Big(\int_0^{\infty} e^{-t}|X_t|^2\,dt\le 8\delta^2\log(1/\delta)\Big)
&\ge \PP\Big(\int_0^{4\log(1/\delta)} e^{-t}|X_t|^2\,dt\le 4\delta^2\log(1/\delta)\\
&\hspace{4.7em}\wedge \int_{4\log(1/\delta)}^{\infty} e^{-t}|X_t|^2\,dt\le 4\delta^2\log(1/\delta)\Big)\\
&\ge \PP\Big(\int_0^{4\log(1/\delta)} e^{-t}|X_t|^2\,dt\le 4\delta^2\log(1/\delta)\Big)\\
&\hspace{2em}\times \PP\Big(\int_{4\log(1/\delta)}^{\infty} e^{-t}|X_t|^2\,dt\le 4\delta^2\log(1/\delta)\Big)\\
&\ge \frac{G(\delta)}{2},
\end{align*}
where the third line uses the Gaussian correlation inequality. Since
\[
    8\delta^2\log(1/\delta)\le (4\delta\log(1/\delta))^2
\]
for all sufficiently small $\delta$, the left-hand side is at most $H(4\delta\log(1/\delta))$.
\end{proof}

The key estimate compares the $L^{\infty}$ and $L^2$ norms of the weighted process
\[
    \wt X_t=w(t)X_t,
\]
where $w$ is given by \cref{lem:2.2}. Since $\wt X_t$ vanishes outside $[-1/2,1/2]$, this is a local comparison.

\begin{lemma}\label{lem:2.4}
For all $L>0$, we have
\[
    \PP\Big(\max_{-1\le t\le1}|\wt X_t| \ge L\Big(\int_{-1}^1 |\wt X_t|^2\,dt\Big)^{1/2}\Big)
    \ll e^{-\Omega(L/\log^2(L+e))}.
\]
\end{lemma}
\begin{proof}
We may assume $L\ge10$, since the remaining values of $L$ can be absorbed into the implicit constant. Set $M=\lfloor L\rfloor$.

We expand $\wt X$ in a Fourier series and separate low and high frequencies. Since $\wt X_t$ vanishes at $\pm1/2$, we may identify it with its $1$-periodic extension and write, for $t\in[-1/2,1/2]$,
\[
    \wt X_t=\sum_{k\in\Z} a_k e^{2\pi i k t},
    \qquad
    a_k=\int_{-1/2}^{1/2} \wt X_t e^{-2\pi i k t}\,dt.
\]
Then
\begin{align*}
\E[|a_k|^2]
&= \int_{-1/2}^{1/2}\int_{-1/2}^{1/2} w(t)w(s)K(t-s)e^{-2\pi i k(t-s)}\,dt\,ds\\
&= \int_{\R}\int_{\R} w(t)w(s)K(t-s)e^{-2\pi i k(t-s)}\,dt\,ds\\
&= \frac{1}{2\pi}\int_{\R}\int_{\R}\int_{\R} w(t)w(s)\widehat K(\xi)e^{i(\xi-2\pi k)(t-s)}\,d\xi\,dt\,ds\\
&= \frac12\int_{\R} \operatorname{sech}(\pi\xi)\,|\widehat w(2\pi k-\xi)|^2\,d\xi.
\end{align*}
Since $\operatorname{sech}(\pi\xi)\ll e^{-\Omega(|\xi|)}$, splitting the integral into the regions $|\xi|\ge |k|$ and $|\xi|\le |k|$ gives
\[
    \E[|a_k|^2]\ll e^{-\Omega(|k|/\log^2(|k|+e))}.
\]
In particular, the high-frequency coefficients are extremely small on average. Moreover,
\[
    \max_{-1\le t\le1}|\wt X_t|\le \sum_k |a_k|,
    \qquad
    \int_{-1}^1 |\wt X_t|^2\,dt = \sum_k |a_k|^2.
\]
Thus, unless the high frequencies dominate the Fourier series, the supremum is controlled by the $L^2$ norm through the low modes alone.

If $\sum_k |a_k|\le2\sum_{|k|\le M}|a_k|$, then by Cauchy--Schwarz,
\begin{align*}
    \max_{-1\le t\le1}|\wt X_t|
    &\le 2\sum_{|k|\le M}|a_k| \\
    &\le 2\sqrt{2M+1}\Big(\sum_{|k|\le M}|a_k|^2\Big)^{1/2} \\
    &\le 2\sqrt{2M+1}\Big(\int_{-1}^1 |\wt X_t|^2\,dt\Big)^{1/2} \\
    &\le L\Big(\int_{-1}^1 |\wt X_t|^2\,dt\Big)^{1/2}.
\end{align*}
Thus it remains to bound the probability that the high frequencies dominate. Let $c>0$ be sufficiently small. Then
\[
    \PP\Big(\sum_{|k|>M}|a_k|\ge e^{-cM/\log^2(M+e)}\Big)
    \ll e^{-cM/\log^2(M+e)}.
\]
Consequently,
\begin{align*}
\PP\Big(\sum_k |a_k|\ge 2\sum_{|k|\le M}|a_k|\Big)
&\ll \PP\Big(\sum_{|k|\le M}|a_k|\le e^{-cM/\log^2(M+e)}\Big)+e^{-cM/\log^2(M+e)}\\
&\ll \PP\Big(|a_0|\le e^{-cM/\log^2(M+e)}\Big)+e^{-cM/\log^2(M+e)}\\
&\ll e^{-cM/\log^2(M+e)}
 \ll e^{-\Omega(L/\log^2(L+e))},
\end{align*}
since $a_0$ is a centered Gaussian with nonzero variance independent of $L$.
\end{proof}

We can now prove the reverse implication between $H(\delta)$ and $G(\delta)$.
\begin{lemma}\label{lem:2.5}
There exist absolute constants $C>0$ and $\delta_0\in(0,1)$ such that whenever $0<\delta<\delta_0$,
\[
    H(\delta)\,\exp\big(-O(\log^2(1/\delta))\big) \le G\big(C\delta\log^4(1/\delta)\big).
\]
\end{lemma}
\begin{proof}
We first enlarge the integration range slightly to the left; by the argument at the end of \cref{lem:2.1}, this costs only an $\exp(-O(\log^2(1/\delta)))$ factor. More precisely, the Gaussian correlation inequality gives
\begin{align*}
\PP\Big(\int_{-1}^{\infty} e^{-t}X_t^2\,dt\le2\delta^2\Big)
&\ge \PP\Big(\int_{-1}^0 e^{-t}X_t^2\,dt\le\delta^2\Big)
    \PP\Big(\int_0^{\infty} e^{-t}X_t^2\,dt\le\delta^2\Big)\\
&\ge \exp\big(-O(\log^2(1/\delta))\big)H(\delta).
\end{align*}
The rest of the proof is devoted to showing that, on this event, the process is very unlikely to develop a narrow spike.
Let $C>0$ be a sufficiently large absolute constant.
Set
\[
    L=\log^4(1/\delta).
\]
To control the interval $[0,L]$, note that the event
\[
    \int_{-1}^{\infty} e^{-t}X_t^2\,dt\le2\delta^2
\]
implies that for every integer $k\ge0$,
\[
    \int_{\R} e^{-t}w(t-k/4)^2|X_t|^2\,dt\le2\delta^2.
\]
Since $w(t-k/4)=1$ whenever $|t-k/4|\le1/4$, applying \cref{lem:2.4} with parameter $L$ for each integer $k\le4L$ and using stationarity gives
\begin{align*}
\PP\Big(\sup_{0\le t\le L} e^{-t/2}|X_t|\le C\delta\log^4(1/\delta)
    \ \wedge\ 
    \int_{-1}^{\infty} e^{-t}X_t^2\,dt\le2\delta^2\Big)
&\ge \PP\Big(\int_{-1}^{\infty} e^{-t}X_t^2\,dt\le2\delta^2\Big) \\
&\qquad - \exp\Big(-\Omega\Big(\frac{\log^4(1/\delta)}{\log^2\log(1/\delta)}\Big)\Big).
\end{align*}

It remains to control the tail. For $\ell\ge0$ we have
\[
    \E\int_{\R} w(t-\ell)^2e^{-t}|X_t|^2\,dt\ll e^{-\ell}.
\]
Applying \cref{lem:2.4} once more,
\[
    \PP\Big(\sup_t w(t-\ell)e^{-t/2}|X_t|\ge e^{-\ell/4}\Big)
    \ll \PP\Big(\int_{\R} w(t-\ell)^2e^{-t}|X_t|^2\,dt\ge e^{-3\ell/4}\Big)+e^{-e^{\Omega(\ell)}}
    \ll e^{-\ell/4}.
\]
Taking $\ell=k/4$ for integers $k\ge100\log(1/\delta)$ and using a union bound, we obtain
\[
    \PP\Big(\sup_{t\ge100\log(1/\delta)} e^{-t/2}|X_t|\le C\delta\log^4(1/\delta)\Big)\ge\frac12
\]
for all sufficiently small $\delta$. Applying the Gaussian correlation inequality once more and combining this with the previous estimate on $[0,L]$, we obtain
\[
    A=\Big\{\sup_{t\ge0} e^{-t/2}|X_t|\le C\delta\log^4(1/\delta)\Big\},
    \qquad
    A_L=\Big\{\sup_{0\le t\le L} e^{-t/2}|X_t|\le C\delta\log^4(1/\delta)\Big\},
\]
and
\[
    B=\Big\{\int_{-1}^{\infty} e^{-t}X_t^2\,dt\le2\delta^2\Big\}.
\]
\begin{align*}
\PP(A\cap B)
&\ge \frac12\PP(A_L\cap B) \\
&\ge \frac12\PP(B) - \exp\Big(-\Omega\Big(\frac{\log^4(1/\delta)}{\log^2\log(1/\delta)}\Big)\Big) \\
&\ge \exp\big(-O(\log^2(1/\delta))\big)H(\delta) \\
&\qquad - \exp\Big(-\Omega\Big(\frac{\log^4(1/\delta)}{\log^2\log(1/\delta)}\Big)\Big).
\end{align*}
To remove the additive error term, it suffices to note that
\[
    H(\delta)\ge \exp(-O(\log^3(1/\delta))).
\]
Indeed, for sufficiently small $\delta$ the event
\[
    \sup_{0\le t\le10\log(1/\delta)} |X_t|\le\delta^2
\]
already implies
\[
    \int_0^{10\log(1/\delta)} e^{-t}|X_t|^2\,dt\le\delta^2/2,
\]
so
\begin{align*}
H(\delta)
&\ge \PP\Big(\sup_{0\le t\le10\log(1/\delta)} |X_t|\le\delta^2\Big)
    \PP\Big(\int_{10\log(1/\delta)}^{\infty} e^{-t}|X_t|^2\,dt\le\delta^2/2\Big) \\
&\ge \frac12\Big(\PP\big(\sup_{0\le t\le1}|X_t|\le\delta^2\big)\Big)^{\lceil 10\log(1/\delta)\rceil} \\
&\ge \exp\big(-O(\log^3(1/\delta))\big),
\end{align*}
where the tail term is handled by Markov's inequality, and the first factor is estimated by decomposing $[0,10\log(1/\delta)]$ into unit intervals and using the argument from the end of \cref{lem:2.1}. This completes the proof.
\end{proof}

\subsection{\texorpdfstring{Completing the proof of \cref{thm:1.2}}{Completing the proof of Theorem 1.2}}\label{subsec:2.3}
Since $H(\delta)$ is an $L^2$ quantity, it can be analyzed by spectral methods. In particular, \cite[Lemma~4]{NP23} implies that
\begin{equation} \label{eq:2.5}
    \log \PP\Big(\int_0^{\infty} e^{-t}X_t^2\,dt<\delta\Big)
    = -\frac{1}{12\pi^2}\log^3(1/\delta)+o\big(\log^3(1/\delta)\big).
\end{equation}
For completeness, we prove \cref{eq:2.5} in \cref{app:A} with a quantitative error term.

\begin{lemma} \label{lem:2.6}
Let $X$ be as in \cref{eq:2.2}. Then for all $\delta\in(0,1/4)$,
\begin{equation} \label{eq:2.6}
    \log \PP\Big(\int_0^{\infty} e^{-t}X_t^2\,dt<\delta\Big)
    = -\frac{1}{12\pi^2}\log^3(1/\delta)
    + O\big(\log^{5/2}(1/\delta)\sqrt{\log\log(1/\delta)}\big).
\end{equation}
\end{lemma}

We now record the corresponding quantitative version of \cref{thm:1.2}.
\begin{theorem}\label{thm:2.7}
For $\delta\in(0,1/4)$,
\[
    \log F(\delta)
    = -\frac{2}{3\pi^2}\log^3(1/\delta)
    + O\big(\log^{5/2}(1/\delta)\sqrt{\log\log(1/\delta)}\big).
\]
\end{theorem}
\begin{proof}
The comparison lemmas sandwich $G(\delta)$ between $H$ evaluated at two nearby polylogarithmic scales, so we begin by recording those two scales explicitly.
Let $C>0$ be the absolute constant from \cref{lem:2.5}.

Set
\[
    \theta_\delta^+=4\delta\log(1/\delta),
    \qquad
    \theta_\delta^- = \frac{\delta}{4C\log^4(1/\delta)}.
\]
By \cref{lem:2.3},
\[
    G(\delta)\le2H(\theta_\delta^+).
\]
Also, for all sufficiently small $\delta$,
\[
    C\theta_\delta^-\log^4(1/\theta_\delta^-)\le\delta,
\]
so monotonicity together with \cref{lem:2.5} gives
\[
    H(\theta_\delta^-)\exp\big(-O(\log^2(1/\delta))\big)\ll G(\delta).
\]
Now \cref{lem:2.6} implies that for $x\in(0,1/4)$,
\[
    \log H(x)
    = -\frac{2}{3\pi^2}\log^3(1/x)
    + O\big(\log^{5/2}(1/x)\sqrt{\log\log(1/x)}\big).
\]
Since $\log(1/\theta_\delta^{\pm})=\log(1/\delta)+O(\log\log(1/\delta))$, the previous two displays yield
\[
    \log G(\delta)
    = -\frac{2}{3\pi^2}\log^3(1/\delta)
    + O\big(\log^{5/2}(1/\delta)\sqrt{\log\log(1/\delta)}\big).
\]
This identifies the asymptotic behavior of $G(\delta)$. To pass back to $F(\delta)$, we appeal once more to \cref{lem:2.1}; after absorbing the fixed constant $G(1)$ into the error term, it gives
\[
    \log F(\delta)\ge \log G(\delta)-O(\log^2(1/\delta))
\]
and
\[
    \log F(\delta/2)\le \log G(\delta)+O(\log^2(1/\delta)\log\log(1/\delta)).
\]
Replacing $\delta/2$ by $\delta$ changes the cubic main term by $O(\log^2(1/\delta))$, which is absorbed by the stated error term. This completes the proof.
\end{proof}
\section[A reduction of Theorem~1.1 to the Gaussian model]{A reduction of \texorpdfstring{\cref{thm:1.1}}{Theorem~1.1} to the Gaussian model} \label{sec:3}
The purpose of this section is to connect $f_n(x)$ with the Gaussian process studied in \cref{sec:2}. Throughout this section, let $(\varepsilon_k)_{k\ge0}$ denote a sequence of independent Rademacher random variables on a probability space $(\Omega,\mc F,\PP)$. We use the logarithmic coordinate $x=\pm e^{-t/n}$, which parametrizes all of $[-1,1]$ except for the single point $0$.

\begin{lemma} \label{lem:3.1}
For $t\ge0$, define
\begin{equation} \label{eq:3.1}
    f_{n,t}^{+}=\frac{1}{\sqrt n}f_n(e^{-t/n}),
    \qquad
    f_{n,t}^{-}=\frac{1}{\sqrt n}f_n(-e^{-t/n}).
\end{equation}
Then
\[
    \snorm{f_n}_\infty
    = \max\Big\{1,\ \sqrt n\sup_{t\ge0}|f_{n,t}^{+}|,\ \sqrt n\sup_{t\ge0}|f_{n,t}^{-}|\Big\}.
\]
\end{lemma}
\begin{proof}
If $x\in(0,1]$, then $x=e^{-t/n}$ for $t=-n\log x\ge0$; if $x\in[-1,0)$, then $x=-e^{-t/n}$ for $t=-n\log|x|\ge0$; and $|f_n(0)|=1$. The claim follows.
\end{proof}

We now compare the profiles from \cref{eq:3.1} with the Gaussian process $Y$ from \cref{eq:2.1}. The key input is the Koml\'os--Major--Tusn\'ady (KMT) strong invariance principle for i.i.d.\ variables with finite exponential moments \cite[Theorem~1]{KMT75,KMT76}; see also \cite[Theorem~1.1]{Chatterjee12} for an alternative proof. The crucial point is that the KMT theorem allows us to replace the random walk $\varepsilon_1,\varepsilon_1+\varepsilon_2,\ldots$ by a suitable rescaling of Brownian motion, with negligible error. An integration-by-parts argument then matches $f_n(x)$ with the corresponding Gaussian model.

\begin{lemma}\label{lem:3.2}
For every $A\ge1$ there exists $C_A>0$ such that the following holds. Let $(\varepsilon_k)_{k\ge0}$ be a sequence of independent Rademacher random variables. One can couple standard Brownian motion $(B_s)_{s\ge0}$ with $(\varepsilon_k)_{k\ge0}$ so that
\[
    \PP\Big(\max_{0\le k\le n}\Big|\sum_{j=0}^k\varepsilon_j-B_k\Big|>C_A\log n\Big)
    \le C_A n^{-A}
\]
for all $n\ge2$.
\end{lemma}

We now use \cref{lem:3.2} to couple $f_n$ with two independent copies of the Gaussian process $Y$.
\begin{lemma}\label{lem:3.4}
There exists an absolute constant $C\ge1$ such that the following holds. Let $(\varepsilon_k)_{0\le k\le n}$ be independent Rademacher random variables. One may couple $(\varepsilon_k)_{0\le k\le n}$ with two independent copies $(Y_{t,1})_{t\ge0}$ and $(Y_{t,2})_{t\ge0}$ of the process $Y$ from \cref{eq:2.1} so that
\[
    \PP\Big(\sup_{t\ge0}\Big|f_n(e^{-t/n})-\sqrt n\,Y_{t,1}\Big|
    + \sup_{t\ge0}\Big|f_n(-e^{-t/n})-\sqrt n\,Y_{t,2}\Big|\ge C\log n\Big)
    \ll n^{-2}.
\]
\end{lemma}
\begin{proof}
Write
\[
    E_n(x)=\frac{f_n(x)+f_n(-x)}{2\sqrt n},
    \qquad
    O_n(x)=\frac{f_n(x)-f_n(-x)}{2\sqrt n}.
\]
Then
\[
    E_n(e^{-t/n})=\frac{1}{\sqrt n}\sum_{\substack{0\le k\le n\\ k\text{ even}}}\varepsilon_k e^{-kt/n},
    \qquad
    O_n(e^{-t/n})=\frac{1}{\sqrt n}\sum_{\substack{0\le k\le n\\ k\text{ odd}}}\varepsilon_k e^{-kt/n}.
\]
The processes $E_n$ and $O_n$ are independent because they depend on disjoint sets of coefficients.

Applying \cref{lem:3.2} to the reindexed even and odd subsequences, we obtain independent standard Brownian motions $(B_s^e)_{0\le s\le1}$ and $(B_s^o)_{0\le s\le1}$ such that, with
\[
    R_s^e=\frac{1}{\sqrt n}\sum_{\substack{1\le k\le ns\\ k\text{ even}}}\varepsilon_k-\frac{1}{\sqrt2}B_s^e,
    \qquad
    R_s^o=\frac{1}{\sqrt n}\sum_{\substack{1\le k\le ns\\ k\text{ odd}}}\varepsilon_k-\frac{1}{\sqrt2}B_s^o,
\]
we have
\[
    \PP\Big(\sup_{0\le s\le1}|R_s^e|+\sup_{0\le s\le1}|R_s^o|\ge C\frac{\log n}{\sqrt n}\Big)\ll n^{-2}
\]
for a suitable absolute constant $C$.

Strictly speaking, \cref{lem:3.2} is stated only at the lattice points $s\in n^{-1}\Z$. However, standard Brownian continuity estimates imply
\[
    \PP\Big(\sup_{|t_1-t_2|\le1/n}|B_{t_1}-B_{t_2}|\ge x\Big)\le n e^{-\Omega(nx^2)},
\]
so the estimate upgrades to all $s\in[0,1]$ at the cost of an error $\ll\sqrt{(\log n)/n}$, which is absorbed into the displayed bound.

Now define
\[
    Y_t^e=2^{-1/2}\int_0^1 e^{-st}\,dB_s^e,
    \qquad
    Y_t^o=2^{-1/2}\int_0^1 e^{-st}\,dB_s^o.
\]
Since
\[
    E_n(e^{-t/n})=\frac{\varepsilon_0}{\sqrt n}+\int_0^1 e^{-st}\,d\Big(\frac{1}{\sqrt2}B_s^e+R_s^e\Big),
    \qquad
    O_n(e^{-t/n})=\int_0^1 e^{-st}\,d\Big(\frac{1}{\sqrt2}B_s^o+R_s^o\Big),
\]
integration by parts gives
\begin{align*}
E_n(e^{-t/n})-Y_t^e
&= \frac{\varepsilon_0}{\sqrt n}+e^{-t}R_1^e+\int_0^1 te^{-st}R_s^e\,ds,\\
O_n(e^{-t/n})-Y_t^o
&= e^{-t}R_1^o+\int_0^1 te^{-st}R_s^o\,ds.
\end{align*}
Since
\[
    \sup_{t\ge0}\int_0^1 te^{-st}\,ds\le1,
\]
it follows that, with probability $1-O(n^{-2})$,
\[
    \sup_{t\ge0}|E_n(e^{-t/n})-Y_t^e| + \sup_{t\ge0}|O_n(e^{-t/n})-Y_t^o| \ll \frac{\log n}{\sqrt n}.
\]
Finally, set
\[
    Y_{t,1}=Y_t^e+Y_t^o,
    \qquad
    Y_{t,2}=Y_t^e-Y_t^o.
\]
Since $Y_t^e$ and $Y_t^o$ are independent centered Gaussian processes, each with covariance $\tfrac12\E[Y_sY_t]$, the processes $(Y_{t,1})_{t\ge0}$ and $(Y_{t,2})_{t\ge0}$ are independent copies of $Y$. Also,
\[
    \frac{1}{\sqrt n}f_n(e^{-t/n})=E_n(e^{-t/n})+O_n(e^{-t/n}),
    \qquad
    \frac{1}{\sqrt n}f_n(-e^{-t/n})=E_n(e^{-t/n})-O_n(e^{-t/n}).
\]
Therefore,
\[
    \sup_{t\ge0}\Big|\frac{1}{\sqrt n}f_n(e^{-t/n})-Y_{t,1}\Big|
    + \sup_{t\ge0}\Big|\frac{1}{\sqrt n}f_n(-e^{-t/n})-Y_{t,2}\Big|
    \ll \frac{\log n}{\sqrt n}
\]
with probability $1-O(n^{-2})$, which is exactly the desired statement.
\end{proof}
\section{Quantitative continuity for the small-ball inverse}
Recall that
\[
    F(\delta)=\PP\Big(\sup_{t\ge0}|Y_t| \le \delta\Big),
    \qquad \delta>0.
\]
Strictly speaking, we have not yet proved that $F$ is continuous and strictly increasing on $(0,\infty)$, but this is a routine exercise in the theory of Gaussian processes. Since $Y$ has continuous sample paths and tends to $0$ at infinity, one may view $Y$ as a centered Gaussian random element of $C_0([0,\infty])$; from there, standard arguments show that $F(\delta)\in(0,1)$ for every $\delta>0$, that $F$ is continuous and strictly increasing, and hence that
\[
    F^{-1}\colon(0,1)\to(0,\infty)
\]
is well-defined. Let
\[
    \Psi(x)=\log F(e^{-x}).
\]
Since $F$ is continuous and strictly increasing, its inverse $F^{-1}$ is also continuous and strictly increasing. The purpose of this section is to quantify how $F$ and $F^{-1}$ change under small multiplicative perturbations; these estimates are crucial for the proof of \cref{thm:1.1}.

\begin{proposition} \label{prop:4.1}
Let $\rho\in(0,1)$. Then there exists $\delta_0\in(0,1)$ such that whenever $\delta\in(0,\delta_0)$ and $t\in[-1,1]$,
\begin{equation} \label{eq:4.1}
    \Bigg|\log \frac{F(\delta e^t)}{F(\delta)} - \frac{2t}{\pi^2}\log^2(1/\delta)\Bigg|
    \le \rho |t|\log^2(1/\delta).
\end{equation}
\end{proposition}
\begin{proof}
By the Gaussian $B$-inequality \cite[Theorem~1]{CFM04}, the function $\Psi$ is concave on $(0,\infty)$. Write
\[
    \Psi(x)=-\frac{2}{3\pi^2}x^3+r(x),
\]
where \cref{thm:1.2} gives $r(x)=o(x^3)$ as $x\to\infty$. Fix $\rho\in(0,1)$. Choose $\varepsilon>0$ sufficiently small, and then choose $x_0$ so large that
\[
    |r(y)|\le\varepsilon y^3
\]
for all $y\ge x_0/2$. For $x\ge x_0$, let $m=\lfloor\varepsilon^{1/2}x\rfloor$. Then $m\in[1,x/2]$, and
\begin{align*}
    \frac{\Psi(x+m)-\Psi(x)}{m}
    &= -\frac{2}{\pi^2}x^2 + O(\varepsilon^{1/2}x^2),\\
    \frac{\Psi(x)-\Psi(x-m)}{m}
    &= -\frac{2}{\pi^2}x^2 + O(\varepsilon^{1/2}x^2).
\end{align*}
After shrinking $\varepsilon$ if necessary, the error terms are bounded by $(\rho/4)x^2$. Concavity of $\Psi$ therefore yields
\begin{align*}
    \Psi(x+1)-\Psi(x)
    &\ge \frac{\Psi(x+m)-\Psi(x)}{m}
     = -\frac{2}{\pi^2}x^2+O(\rho x^2),\\
    \Psi(x)-\Psi(x-1)
    &\le \frac{\Psi(x)-\Psi(x-m)}{m}
     = -\frac{2}{\pi^2}x^2+O(\rho x^2).
\end{align*}
For $t\in(0,1]$, concavity again gives
\[
    t\big(\Psi(x+1)-\Psi(x)\big)
    \le \Psi(x)-\Psi(x-t)
    \le t\big(\Psi(x)-\Psi(x-1)\big),
\]
and hence
\[
    \Psi(x-t)-\Psi(x)=\frac{2t}{\pi^2}x^2+O(\rho t x^2).
\]
Since
\[
    \log\frac{F(\delta e^t)}{F(\delta)}
    = \Psi\Big(\log\frac{1}{\delta}-t\Big)-\Psi\Big(\log\frac{1}{\delta}\Big),
\]
this proves \cref{eq:4.1} for $t\in(0,1]$. The case $t\in[-1,0)$ follows by applying the same estimate to $\delta e^t$ in place of $\delta$ and observing that
\[
    \log^2\frac{1}{\delta e^t}=\log^2\frac{1}{\delta}+O\Big(\log\frac{1}{\delta}\Big).
\]
Since $|t|\le1$, the resulting change in the main term is absorbed by the right-hand side of \cref{eq:4.1} once $\delta$ is sufficiently small.
\end{proof}

\begin{corollary}\label{cor:4.2}
Let $\rho\in(0,1)$. Then there exist $\delta_0\in(0,1)$ and $\tau_0\in(0,1)$ such that whenever $\delta\in(0,\delta_0)$ and $|\tau|\le\tau_0$,
\[
    \Bigg|\log\frac{F(\delta(1+\tau))}{F(\delta)} - \frac{2\tau}{\pi^2}\log^2(1/\delta)\Bigg|
    \le \rho |\tau|\log^2(1/\delta).
\]
\end{corollary}
\begin{proof}
By \cref{prop:4.1}, after shrinking $\delta_0$ if necessary,
\[
    \Bigg|\log\frac{F(\delta e^t)}{F(\delta)} - \frac{2t}{\pi^2}\log^2(1/\delta)\Bigg|
    \le \frac{\rho}{4}|t|\log^2(1/\delta)
\]
whenever $0<\delta<\delta_0$ and $t\in[-1,1]$. Now set $t=\log(1+\tau)$. For $|\tau|\le1/2$,
\[
    |t|\le2|\tau|,
    \qquad
    |t-\tau|\le2\tau^2.
\]
After shrinking $\tau_0\in(0,1/2]$ if necessary, we also have
\[
    \frac{2}{\pi^2}|t-\tau|\le\frac{\rho}{2}|\tau|
\]
for all $|\tau|\le\tau_0$. The triangle inequality now gives
\[
    \Bigg|\log\frac{F(\delta(1+\tau))}{F(\delta)} - \frac{2\tau}{\pi^2}\log^2(1/\delta)\Bigg|
    \le \rho|\tau|\log^2(1/\delta). \qedhere
\]
\end{proof}

The next corollary is the inverse form needed later. It says that if one perturbs the value of $F$ multiplicatively by $e^{\sigma_n}$, with $|\sigma_n|=o(\log^2(1/\delta_n))$, then the corresponding perturbation of $F^{-1}$ has relative size $\sigma_n/\log^2(1/\delta_n)$.
\begin{corollary} \label{cor:4.3}
Let $\eta>0$, let $(y_n)$ be a sequence in $(0,1)$ with $y_n\downarrow0$, let $(\sigma_n)$ be a sequence of real numbers such that $y_n e^{\sigma_n}\in(0,1)$ for all $n$, and set $\delta_n=F^{-1}(y_n)$. Assume that
\[
    \frac{|\sigma_n|}{\log^2(1/\delta_n)}\to0.
\]
Then for all sufficiently large $n$,
\begin{equation} \label{eq:4.2}
    \Bigg|\frac{F^{-1}(y_n e^{\sigma_n})}{F^{-1}(y_n)} - \Big(1+\frac{\pi^2\sigma_n}{2\log^2(1/\delta_n)}\Big)\Bigg|
    \le \eta\frac{|\sigma_n|}{\log^2(1/\delta_n)}.
\end{equation}
\end{corollary}
\begin{proof}
Let
\[
    1+\tau_n = \frac{F^{-1}(y_n e^{\sigma_n})}{F^{-1}(y_n)},
\]
so that
\[
    \sigma_n = \log\frac{F(\delta_n(1+\tau_n))}{F(\delta_n)}.
\]
Choose $\rho\in(0,\pi^{-2}]$ so that $\pi^4\rho/2\le\eta$. Since $|\sigma_n|/\log^2(1/\delta_n)\to0$, \cref{cor:4.2} applied at the fixed perturbations $\pm\tau_0$ shows that $|\tau_n|\le\tau_0$ eventually. For such $n$,
\[
    \Big|\sigma_n-\frac{2}{\pi^2}\tau_n\log^2(1/\delta_n)\Big|
    \le \rho|\tau_n|\log^2(1/\delta_n).
\]
Since $\rho\le\pi^{-2}$, this implies
\[
    \frac{1}{\pi^2}|\tau_n|\log^2(1/\delta_n)\le|\sigma_n|,
\]
and therefore
\[
    \Big|\tau_n-\frac{\pi^2\sigma_n}{2\log^2(1/\delta_n)}\Big|
    \le \frac{\pi^2\rho}{2}|\tau_n|
    \le \eta\frac{|\sigma_n|}{\log^2(1/\delta_n)}.
\]
This is exactly \cref{eq:4.2}.
\end{proof}
\section[Proving Theorem~1.1]{Proving \texorpdfstring{\cref{thm:1.1}}{Theorem~1.1}} \label{sec:5}
Throughout this section, $N$ denotes a dyadic scale, that is, $N=2^k$. We write
\begin{equation} \label{eq:5.1}
    s_n=(\log\log n)^{1/3},
    \qquad
    p_n=\log^{-1/2}n,
    \qquad
    b_n=F^{-1}(p_n).
\end{equation}
Then \cref{thm:1.1} is equivalent to
\begin{equation} \label{eq:5.2}
    \liminf_{n\to\infty}\frac{\snorm{f_n}_\infty}{\sqrt n\,b_n}=1
    \qquad \text{almost surely.}
\end{equation}
The scale $b_n$ is defined by $F(b_n)=p_n$. We first record its size and its stability on dyadic blocks. We then prove the lower bound in \cref{eq:5.2} by passing from a dyadic block to a geometric mesh and using the coupling from \cref{sec:3}. Finally, we prove the upper bound by splitting $f_{N_{j+1}}$ into an old part and an independent fresh block.

\subsection{Asymptotics and stability of the scale sequence} \label{subsec:5.1}
We begin with the two properties of $b_n$ needed below: its size and its stability under small relative changes in $n$.

\begin{lemma} \label{lem:5.1}
Let $(b_n)$ be defined by \cref{eq:5.1}. Then, for all sufficiently large $n$,
\begin{equation} \label{eq:5.3}
    \log(1/b_n)=\Big(\Big(\frac{3\pi^2}{4}\Big)^{1/3}+o(1)\Big)s_n.
\end{equation}
\end{lemma}
\begin{proof}
Since $F(b_n)=p_n=\log^{-1/2}n$, we have $\log p_n=-\frac12\log\log n=-\frac12 s_n^3$. Applying \cref{thm:1.2} with $\delta=b_n$ gives
\[
    \frac{2}{3\pi^2}\log^3(1/b_n)=\frac12 s_n^3+o\big(\log^3(1/b_n)\big).
\]
In particular, $\log(1/b_n)\asymp s_n$, since otherwise the two cubic terms could not balance. Substituting this back into the displayed identity yields
\[
    \log^3(1/b_n)=\frac{3\pi^2}{4}s_n^3+o(s_n^3),
\]
and taking cube roots proves \cref{eq:5.3}.
\end{proof}

\begin{lemma}\label{lem:5.2}
There exists an absolute constant $C>0$ such that for all sufficiently large $N$, all $\Delta\in(0,1/2]$, and all $m,n\in[N,2N]$ with $|m-n|\le\Delta N$, one has
\begin{equation} \label{eq:5.4}
    \Big|\frac{b_n}{b_m}-1\Big|\le C\frac{\Delta}{(\log N)(\log\log N)^{2/3}}.
\end{equation}
\end{lemma}
\begin{proof}
Set $\sigma_{m,n}=\log(p_n/p_m)$, so that $b_n=F^{-1}(p_m e^{\sigma_{m,n}})$. Since
\[
    \sigma_{m,n}=\frac12\log\frac{\log m}{\log n},
\]
the mean-value theorem gives $|\sigma_{m,n}|\ll \Delta/\log N$ uniformly for $m,n\in[N,2N]$ with $|m-n|\le\Delta N$. Also, \cref{lem:5.1} shows that $\log(1/b_m)\asymp s_N\asymp(\log\log N)^{1/3}$ uniformly on $[N,2N]$, so
\[
    \frac{|\sigma_{m,n}|}{\log^2(1/b_m)}
    \ll \frac{\Delta}{(\log N)(\log\log N)^{2/3}}
    \to0.
\]
Suppose that \cref{eq:5.4} fails. Then there are sequences $N_i\to\infty$, $\Delta_i\in(0,1/2]$, and $m_i,n_i\in[N_i,2N_i]$ with $|m_i-n_i|\le\Delta_iN_i$ such that
\[
    \Big|\frac{b_{n_i}}{b_{m_i}}-1\Big|
    > i\,\frac{\Delta_i}{(\log N_i)(\log\log N_i)^{2/3}}.
\]
Set $y_i=p_{m_i}$, $\delta_i=b_{m_i}$, and $\sigma_i=\sigma_{m_i,n_i}$. Then $b_{n_i}=F^{-1}(y_i e^{\sigma_i})$, while the previous display shows that
\[
    \frac{|\sigma_i|}{\log^2(1/\delta_i)}\to0.
\]
Passing to a subsequence, we may also assume that $y_i\downarrow0$, since $y_i\to0$.
Applying \cref{cor:4.3} with $\eta=1$ to these sequences gives
\[
    \Big|\frac{b_{n_i}}{b_{m_i}}-1\Big|
    \ll \frac{|\sigma_i|}{\log^2(1/\delta_i)}
    \ll \frac{\Delta_i}{(\log N_i)(\log\log N_i)^{2/3}},
\]
contradicting the choice of the counterexample sequence. This proves \cref{eq:5.4}.
\end{proof}

\subsection{\texorpdfstring{The lower bound $\liminf\ge1$}{The lower bound liminf >= 1}}\label{subsec:5.2}
Fix a dyadic block $[N,2N]$. We pass to a geometric mesh $\mc I_N$: \cref{lem:5.3} controls the error between nearby indices, and the mesh is sparse enough that the Gaussian small-ball probabilities are summable over $N$.

\begin{lemma} \label{lem:5.3}
There exists an absolute constant $C>0$ such that the following holds. Let $N$ be sufficiently large and $\Delta\in[N^{-1},e^{-1}]$. Then
\begin{equation} \label{eq:5.5}
    \PP\bigg(\max_{\substack{n,m\in[N,2N]\\ |n-m|\le\Delta N}}\max_{x\in[-1,1]} |f_n(x)-f_m(x)|
    \ge C\sqrt{\Delta N}\sqrt{\log(1/\Delta)+\log\log N}\bigg)
    \ll \log^{-2}N.
\end{equation}
\end{lemma}
\begin{proof}
For $m<n$ and $x\in[0,1]$, summation by parts gives
\[
    f_n(x)-f_m(x)
    =(S_n-S_m)x^n+\sum_{k=m+1}^{n-1}(S_k-S_m)(x^k-x^{k+1}),
\]
where $S_k=\sum_{j=0}^k\varepsilon_j$. Since
\[
    x^n+\sum_{k=m+1}^{n-1}(x^k-x^{k+1})=x^{m+1}\le1,
\]
we obtain
\[
    \max_{x\in[0,1]}|f_n(x)-f_m(x)|\le \max_{m\le k\le n}|S_k-S_m|.
\]
The same bound on $[-1,0]$ follows after replacing $\varepsilon_k$ with $(-1)^k\varepsilon_k$, which does not change the law. Therefore, for every $t>0$,
\begin{align*}
\PP\bigg(\max_{\substack{n,m\in[N,2N]\\ |n-m|\le\Delta N}}\max_{x\in[-1,1]} |f_n(x)-f_m(x)|\ge t\bigg)
\le 2\PP\bigg(\max_{\substack{n,m\in[N,2N]\\ |n-m|\le\Delta N}}\max_{m\le k\le n}|S_k-S_m|\ge t\bigg).
\end{align*}
Let $T=\lceil\Delta N\rceil$ and let $I_j=[N+jT,N+(j+2)T]\cap\Z$. Any pair $m,n\in[N,2N]$ with $|n-m|\le\Delta N$ lies in some $I_j$, and there are $O(\Delta^{-1})$ such intervals. By stationarity of the increments,
\[
    \PP\bigg(\max_{\substack{n,m\in[N,2N]\\ |n-m|\le\Delta N}}\max_{m\le k\le n}|S_k-S_m|\ge t\bigg)
    \ll \Delta^{-1}\PP\Big(\max_{0\le r\le2T}|S_r|\ge t/2\Big).
\]
The reflection principle and Hoeffding's inequality then give
\[
    \PP\Big(\max_{0\le r\le2T}|S_r|\ge u\Big)\ll e^{-u^2/(8T)}.
\]
Taking $u=t/2$ with
\[
    t=C\sqrt{\Delta N}\sqrt{\log(1/\Delta)+\log\log N}
\]
and using $T\le2\Delta N$, we find
\[
    \Delta^{-1}\PP\Big(\max_{0\le r\le2T}|S_r|\ge t/2\Big)
    \ll \Delta^{-1}e^{-cC^2(\log(1/\Delta)+\log\log N)}
\]
for an absolute constant $c>0$. Choosing $C$ so that $cC^2\ge4$, the right-hand side is
\[
    \ll \Delta^3\log^{-4}N \ll \log^{-2}N,
\]
which proves \cref{eq:5.5}.
\end{proof}

As a consequence, we may work on a probability-one event on which $f_n$ changes little whenever $n$ varies by at most $\Delta_N N$ inside a dyadic block.
\begin{corollary} \label{cor:5.4}
Let $C>0$ be the absolute constant from \cref{lem:5.3}. Fix $A>0$ and, for each dyadic $N=2^k$, set $\Delta_N=e^{-As_N}$. Then, almost surely, for all sufficiently large dyadic $N$,
\[
    \max_{\substack{n,m\in[N,2N]\\ |n-m|\le\Delta_N N}}\max_{x\in[-1,1]} |f_n(x)-f_m(x)|
    < C\sqrt{\Delta_N N}\sqrt{\log(1/\Delta_N)+\log\log N}.
\]
\end{corollary}
\begin{proof}
Since $s_N\to\infty$ and $s_N=o(\log N)$, we have $\Delta_N\in[N^{-1},e^{-1}]$ for all sufficiently large dyadic $N$. Thus \cref{lem:5.3} applies with $\Delta=\Delta_N$, and the conclusion follows from Borel--Cantelli because
\[
    \sum_{k\ge1}\log^{-2}(2^k)<\infty. \qedhere
\]
\end{proof}

\begin{proposition}\label{prop:5.5}
Almost surely,
\[
    \liminf_{n\to\infty}\frac{\snorm{f_n}_\infty}{\sqrt n\,b_n}\ge1.
\]
\end{proposition}
\begin{proof}
Fix $\eta\in(0,1)$, let $C>0$ be the constant from \cref{cor:5.4}, and set
\[
    \mc A_n=\{\snorm{f_n}_\infty\le(1-\eta)\sqrt n\,b_n\}.
\]
We will show that $\mc A_n$ occurs only finitely often almost surely.

Choose $A>2(3\pi^2/4)^{1/3}$. For each dyadic $N=2^k$, define
\[
    \Delta_N=e^{-As_N},
    \qquad
    E_N=C\sqrt{\Delta_N N}\sqrt{\log(1/\Delta_N)+\log\log N}.
\]
Let $\mc I_N=\{m_0,\dots,m_{j_N}\}\subset[N,2N]$ be the geometric mesh defined by $m_0=N$ and $m_{j+1}=\lfloor(1+\Delta_N)m_j\rfloor$ until the first index exceeding $2N$, and adjoin $2N$ if necessary. For large $N$, we have $m_{j+1}\ge(1+\Delta_N/2)m_j$, so $|\mc I_N|\ll\Delta_N^{-1}$. Also let
\[
    \mc M_N=
    \Big\{\max_{\substack{m,n\in[N,2N]\\ |m-n|\le\Delta_N N}} \snorm{f_n-f_m}_\infty < E_N\Big\}.
\]
By \cref{cor:5.4}, almost surely $\mc M_N$ holds for all sufficiently large dyadic $N$.

We first reduce the problem from the whole block $[N,2N]$ to the mesh $\mc I_N$. If $m,n\in[N,2N]$ and $|m-n|\le\Delta_N N$, then
\[
    \Big|\sqrt{n/m}-1\Big|\le\Delta_N,
\]
while \cref{lem:5.2} gives
\[
    b_n\le\Big(1+C\frac{\Delta_N}{(\log N)(\log\log N)^{2/3}}\Big)b_m
\]
for an absolute constant $C$. Since $\Delta_N\to0$, for large $N$ this implies
\[
    \sqrt n\,b_n\le(1+\Delta_N)^2\sqrt m\,b_m.
\]
Consequently,
\[
    (1-\eta)\sqrt n\,b_n\le(1-3\eta/4)\sqrt m\,b_m.
\]
Moreover, \cref{lem:5.1} gives
\[
    b_N\ge \exp\Big(-\Big((3\pi^2/4)^{1/3}+o(1)\Big)s_N\Big),
\]
so
\[
    \frac{E_N}{\sqrt N\,b_N}
    \le C\exp\Big(-\Big(\frac{A}{2}-(3\pi^2/4)^{1/3}+o(1)\Big)s_N\Big)\sqrt{As_N+s_N^3}
    \to0.
\]
This is where the choice of $A$ is used. Since $b_{2N}/b_N\to1$ by \cref{cor:4.3}, it follows that
\[
    E_N\le(\eta/4)\sqrt m\,b_m
\]
uniformly for $m\in[N,2N]$, once $N$ is large enough. Consequently, if $\mc A_n\cap\mc M_N$ occurs for some $n\in[N,2N]$ and $m\in\mc I_N$ satisfies $|m-n|\le\Delta_N N$, then
\[
    \snorm{f_m}_\infty \le \snorm{f_n}_\infty+\snorm{f_m-f_n}_\infty
    \le \Big(1-\frac{\eta}{2}\Big)\sqrt m\,b_m.
\]
Such an $m$ always exists for large $N$: every $n\in[N,2N]$ lies between two consecutive mesh points. If $m_j,m_{j+1}\in\mc I_N$ are consecutive and $m_{j+1}<2N$, then
\[
    m_{j+1}-m_j=\lfloor(1+\Delta_N)m_j\rfloor-m_j\le\Delta_Nm_j\le2\Delta_NN.
\]
If the upper mesh point is $2N$, then
\[
    2N-m_j<\lfloor(1+\Delta_N)m_j\rfloor-m_j\le\Delta_Nm_j\le2\Delta_NN.
\]
Hence every mesh gap is $<2\Delta_NN$, so choosing the nearer endpoint gives $|m-n|\le\Delta_NN$.
Therefore,
\[
    \bigcup_{n\in[N,2N]} (\mc A_n\cap\mc M_N)
    \subset \bigcup_{m\in\mc I_N} \mc B_m,
    \qquad
    \mc B_m=\Big\{\snorm{f_m}_\infty\le\Big(1-\frac{\eta}{2}\Big)\sqrt m\,b_m\Big\}.
\]

It remains to estimate the mesh events $\mc B_m$. Fix $m\in\mc I_N$, and choose a coupling as in \cref{lem:3.4}. Let $\mc E_m$ be the event that
\[
    \sup_{t\ge0}|f_{m,t}^{+}-Y^{m,+}(t)| + \sup_{t\ge0}|f_{m,t}^{-}-Y^{m,-}(t)| \le \frac{\eta}{4}b_m,
\]
where $(Y^{m,+}(t))_{t\ge0}$ and $(Y^{m,-}(t))_{t\ge0}$ are independent copies of $Y$. Since $\log(1/b_m)\asymp s_m$ by \cref{lem:5.1}, the estimate in \cref{lem:3.4} implies that
\[
    \PP(\mc E_m^c)=O(m^{-2}).
\]
On $\mc B_m\cap\mc E_m$, \cref{lem:3.1} gives
\[
    \sup_{t\ge0}|f_{m,t}^{\pm}|\le\Big(1-\frac{\eta}{2}\Big)b_m,
\]
and hence
\[
    \sup_{t\ge0}|Y^{m,\pm}(t)|\le\Big(1-\frac{\eta}{4}\Big)b_m.
\]
Thus
\[
    \mc B_m\cap\mc E_m\subset \mc H_m,
\]
where
\[
    \mc H_m=
    \Big\{
        \sup_{t\ge0}|Y^{m,+}(t)|\le\Big(1-\frac{\eta}{4}\Big)b_m,
        \ \sup_{t\ge0}|Y^{m,-}(t)|\le\Big(1-\frac{\eta}{4}\Big)b_m
    \Big\}.
\]
Therefore
\[
    \PP(\mc B_m)\le \PP(\mc H_m)+\PP(\mc E_m^c).
\]
Since the two Gaussian copies are independent, \cref{prop:4.1} yields
\[
    \PP(\mc H_m)
    = F\Big(\Big(1-\frac{\eta}{4}\Big)b_m\Big)^2
    \le F(b_m)^2\exp\big(-c_\eta\log^2(1/b_m)\big)
\]
for some $c_\eta>0$. Now $F(b_m)=p_m$, while $m\in[N,2N]$ implies $\log m\asymp\log N$ and $s_m\asymp s_N$. Hence
\[
    \sum_{m\in\mc I_N}\PP(\mc B_m)
    \ll |\mc I_N|\frac{1}{\log N}\exp(-c_\eta' s_N^2)+\sum_{m\in\mc I_N}m^{-2}
    \ll \frac{1}{\log N}\exp(As_N-c_\eta' s_N^2)+\Delta_N^{-1}N^{-2}.
\]
Both terms are summable over dyadic $N=2^k$. By Borel--Cantelli, almost surely the events $\mc B_m$ occur only finitely often over the union of all dyadic meshes. Since $\mc M_N$ also holds for all sufficiently large dyadic $N$, the preceding inclusion shows that $\mc A_n$ occurs only finitely often almost surely. This proves the proposition.
\end{proof}

\subsection{\texorpdfstring{The upper bound $\liminf\le1$}{The upper bound liminf <= 1}}\label{subsec:5.3}
For the upper bound, we pass to a sparse sequence $N_j$ and write
\[
    f_{N_{j+1}}(x)=f_{N_j}(x)+x^{N_j+1}g_j(x).
\]
The fresh block $g_j$ is an independent Littlewood polynomial, so the Gaussian small-ball estimate and Borel--Cantelli show that it is small infinitely often almost surely, while the old part $f_{N_j}$ is negligible on the larger scale $\sqrt{N_{j+1}}\,b_{N_{j+1}}$.

\begin{proposition} \label{prop:5.6}
Almost surely,
\begin{equation} \label{eq:5.6}
    \liminf_{n\to\infty}\frac{\snorm{f_n}_\infty}{\sqrt n\,b_n}\le1.
\end{equation}
\end{proposition}
\begin{proof}
Let $\eta>0$. Choose $A>2(3\pi^2/4)^{1/3}$. For $j\ge1$, set
\[
    N_j=\lceil e^{Aj\log^{1/3}(j+1)}\rceil,
    \qquad
    M_j=N_{j+1}-N_j-1.
\]
Let
\[
    g_j(x)=\sum_{k=0}^{M_j}\varepsilon_{N_j+1+k}x^k.
\]
Then $g_j$ is a degree-$M_j$ Littlewood polynomial depending only on the block $(N_j,N_{j+1}]$, and
\begin{equation} \label{eq:5.7}
    f_{N_{j+1}}(x)=f_{N_j}(x)+x^{N_j+1}g_j(x).
\end{equation}
Let
\[
    \mc A_j' = \{\snorm{g_j}_\infty\le(1+\eta/4)\sqrt{M_j}\,b_{M_j}\}.
\]
We first show that these events occur infinitely often almost surely. For $t\ge0$, set
\[
    g_{j,t}^{+}=M_j^{-1/2}g_j(e^{-t/M_j}),
    \qquad
    g_{j,t}^{-}=M_j^{-1/2}g_j(-e^{-t/M_j}).
\]
By the analogue of \cref{lem:3.1},
\[
    \snorm{g_j}_\infty
    = \max\Big\{1,\ \sqrt{M_j}\sup_{t\ge0}|g_{j,t}^{+}|,\ \sqrt{M_j}\sup_{t\ge0}|g_{j,t}^{-}|\Big\}.
\]
Since $g_j$ has the same law as a degree-$M_j$ Littlewood polynomial, the proof of \cref{lem:3.4}, together with the quantitative estimate in \cref{lem:3.2}, gives a coupling with independent copies $(Y_{j,t}^{+})_{t\ge0}$ and $(Y_{j,t}^{-})_{t\ge0}$ of $Y$ such that
\[
    \PP\Big(\sup_{t\ge0}|g_{j,t}^{+}-Y_{j,t}^{+}| + \sup_{t\ge0}|g_{j,t}^{-}-Y_{j,t}^{-}| > C\frac{\log M_j}{\sqrt{M_j}}\Big)\ll M_j^{-2}
\]
for some absolute constant $C>0$. By \cref{lem:5.1}, $(\log M_j)/\sqrt{M_j}=o(b_{M_j})$, so for all sufficiently large $j$ the exceptional event has probability $O(M_j^{-2})$ even with threshold $(\eta/4)b_{M_j}$. Let $\mc E_j$ denote the complementary good event. Also let $\mc H_j$ be the event that both $\sup_{t\ge0}|Y_{j,t}^{+}|$ and $\sup_{t\ge0}|Y_{j,t}^{-}|$ are at most $b_{M_j}$. On $\mc H_j\cap\mc E_j$ we have $\sup_{t\ge0}|g_{j,t}^{\pm}|\le(1+\eta/4)b_{M_j}$, and since $\sqrt{M_j}b_{M_j}\to\infty$, the identity above implies $\mc H_j\cap\mc E_j\subset\mc A_j'$ eventually. Hence
\[
    \PP(\mc A_j')\ge \PP(\mc H_j)-\PP(\mc E_j^c)=F(b_{M_j})^2-O(M_j^{-2})=\frac{1}{\log M_j}-O(M_j^{-2}).
\]
In particular, $\PP(\mc A_j')\ge(2\log M_j)^{-1}$ for all sufficiently large $j$. Since $N_j/N_{j+1}\to0$, we have $M_j\sim N_{j+1}$ and therefore $\log M_j\asymp j\log^{1/3}j$, so
\[
    \sum_{j\ge1}\PP(\mc A_j')=\infty.
\]
The events $\mc A_j'$ are independent because they depend on disjoint coefficient blocks, and hence the second Borel--Cantelli lemma implies that $\mc A_j'$ occurs infinitely often almost surely.

We next show that the old part is negligible:
\begin{equation} \label{eq:5.8}
    \frac{\snorm{f_{N_j}}_\infty}{\sqrt{N_{j+1}}\,b_{N_{j+1}}}\to0
    \qquad \text{almost surely.}
\end{equation}
Indeed, by \cref{eq:1.1}, almost surely,
\[
    \frac{\snorm{f_{N_j}}_\infty}{\sqrt{N_{j+1}}\,b_{N_{j+1}}}
    \ll \sqrt{\frac{N_j}{N_{j+1}}}\,\frac{\sqrt{\log\log N_j}}{b_{N_{j+1}}}.
\]
Now $\log\log N_j=\log j+O(\log\log j)$, so $\sqrt{\log\log N_j}=\exp(o(\log^{1/3}j))$. Also,
\[
    \sqrt{\frac{N_j}{N_{j+1}}}=\exp\big(-(A/2+o(1))\log^{1/3}j\big),
\]
while \cref{lem:5.1} gives
\[
    b_{N_{j+1}}^{-1}
    = \exp\Big(\Big((3\pi^2/4)^{1/3}+o(1)\Big)\log^{1/3}j\Big).
\]
Therefore,
\[
    \frac{\snorm{f_{N_j}}_\infty}{\sqrt{N_{j+1}}\,b_{N_{j+1}}}
    \ll \exp\Big(-\Big(\frac{A}{2}-(3\pi^2/4)^{1/3}+o(1)\Big)\log^{1/3}j\Big)\to0,
\]
since $A/2>(3\pi^2/4)^{1/3}$. This proves \cref{eq:5.8}.

Also $M_j/N_{j+1}\to1$, so $\sqrt{M_j/N_{j+1}}\to1$. Applying \cref{cor:4.3} with $y_j=p_{N_{j+1}}$ and $\sigma_j=\log(p_{M_j}/p_{N_{j+1}})$, we have $\sigma_j\to0$ and
\[
    |\sigma_j|=
    \frac12\Big|\log\frac{\log N_{j+1}}{\log M_j}\Big|
    = o\big((\log\log N_{j+1})^{2/3}\big)
    \asymp o\big(\log^2(1/b_{N_{j+1}})\big),
\]
so
\begin{equation} \label{eq:5.9}
    \frac{b_{M_j}}{b_{N_{j+1}}}=1+o(1).
\end{equation}

On the probability-one event where $\mc A_j'$ occurs infinitely often and \cref{eq:5.8} holds, we obtain from \cref{eq:5.7,eq:5.9} that for infinitely many $j$,
\begin{align*}
    \snorm{f_{N_{j+1}}}_\infty
    &\le \snorm{f_{N_j}}_\infty+\snorm{g_j}_\infty\\
    &\le o\big(\sqrt{N_{j+1}}\,b_{N_{j+1}}\big)+(1+\eta/4)\sqrt{M_j}\,b_{M_j}\\
    &= \big(1+\eta/4+o(1)\big)\sqrt{N_{j+1}}\,b_{N_{j+1}}.
\end{align*}
Hence, almost surely,
\[
    \liminf_{n\to\infty}\frac{\snorm{f_n}_\infty}{\sqrt n\,b_n}\le1+\eta.
\]
Since $\eta>0$ was arbitrary, \cref{eq:5.6} follows.
\end{proof}

\begin{proof}[Proof of \cref{thm:1.1}]
The lower bound is precisely \cref{prop:5.5}, and the upper bound is precisely \cref{prop:5.6}. Combining them proves \cref{thm:1.1}.
\end{proof}
\appendix
\section{\texorpdfstring{Proof of \cref{lem:2.6}: calculating the $L^2$ small-ball probability}{Proof of Lemma 2.6: calculating the L2 small-ball probability}} \label{app:A}
In this appendix we prove \cref{lem:2.6} by counting the eigenvalues of a covariance operator. Let $X$ be as in \cref{eq:2.2}, and set
\[
    I=\int_0^{\infty} e^{-t}X_t^2\,dt.
\]
Let $T$ be the covariance operator of $e^{-t/2}X_t$. Thus $T\colon L^2(0,\infty)\to L^2(0,\infty)$ is given by
\begin{equation} \label{eq:A.1}
    (Tf)(s)=\int_0^{\infty} e^{-(s+t)/2}K(s-t)f(t)\,dt,
\end{equation}
where $K$ is the covariance kernel from \cref{eq:2.3}. The operator $T$ is compact, self-adjoint, positive, and trace class; moreover, $\Tr(T)=1/2$ by a direct computation. Writing $\lambda_1\ge\lambda_2\ge\cdots\ge0$ for its eigenvalues and choosing an orthonormal basis of eigenvectors, the Karhunen--Lo\`eve expansion gives, almost surely,
\begin{equation} \label{eq:A.2}
    I=\sum_{k\ge1} \lambda_k \xi_k^2,
\end{equation}
where the $\xi_k$ are i.i.d.\ standard Gaussian random variables.

It therefore remains to understand the eigenvalues of $T$. If $A$ is compact, write $s_1(A)\ge s_2(A)\ge\cdots\ge0$ for its singular values, and if $A$ is compact and positive, write $\lambda_1(A)\ge\lambda_2(A)\ge\cdots\ge0$ for its eigenvalues. For $\tau>0$, define
\begin{equation} \label{eq:A.3}
    N(\tau,A)=\#\{k\ge1:s_k(A)\ge\tau\},
    \qquad
    \Lambda(\tau,A)=\#\{k\ge1:\lambda_k(A)\ge\tau\}.
\end{equation}
If $A$ is positive then $N(\tau,A)=\Lambda(\tau,A)$. In particular, $\Lambda(\tau,T)$ counts the eigenvalues of $T$ above the threshold $\tau$. The main input is the quantitative estimate for $\Lambda(\tau,T)$ in \cref{prop:A.4}, proved by combining Laptev's block decomposition \cite{Laptev74} with the estimate from \cite{KRD21} recorded in \cref{lem:A.2}. We then deduce \cref{lem:2.6}.

\subsection{Counting the eigenvalues of \texorpdfstring{$T$}{T}}\label{subsec:A.1}
\begin{lemma} \label{lem:A.1}
The operator $T$ is unitarily equivalent to the operator $\widetilde T$ on $L^2(0,\infty)$ given by
\[
    (\widetilde T f)(s)=\int_0^{\infty} \frac{2}{e^{2s}+e^{2t}}f(t)\,dt.
\]
\end{lemma}
\begin{proof}
Let $U\colon L^2(0,\infty)\to L^2(0,\infty)$ be the unitary dilation $(Uf)(s)=2^{-1/2}f(s/2)$. A direct computation in \cref{eq:A.1}, followed by the change of variables $t\mapsto2t$, gives
\[
    (U^{-1}TUf)(s)=\int_0^{\infty} \frac{2}{e^{2s}+e^{2t}}f(t)\,dt,
\]
which proves the claim.
\end{proof}

For $u,a>0$, let $\mc G_{u,a}\colon L^2(0,a)\to L^2(0,a)$ be the compact operator
\begin{equation} \label{eq:A.4}
    (\mc G_{u,a}f)(x)=\int_0^a \frac{\sin(u(x-y))}{\pi(x-y)}f(y)\,dy.
\end{equation}
Since $\sin(uz)/(\pi z)$ is the inverse Fourier transform of $\one_{[-u,u]}$, Plancherel expresses the quadratic form of $\mc G_{u,a}$ in terms of the Fourier transform of $f\one_{(0,a)}$. In particular, for all $f\in L^2(0,a)$,
\[
    \langle \mc G_{u,a}f,f\rangle
    =\frac{1}{2\pi}\int_{-u}^u \Big|\int_0^a e^{-i\xi x}f(x)\,dx\Big|^2\,d\xi\ge0,
\]
so $\mc G_{u,a}$ is positive. The next lemma estimates $\Lambda(\tau,\mc G_{u,a})$.

\begin{lemma} \label{lem:A.2}
For $\tau\in(0,1/4)$,
\[
    \Lambda(\tau,\mc G_{u,a})=\frac{au}{\pi}+O\big(\log(au+e)\log(1/\tau)\big).
\]
\end{lemma}
\begin{proof}
Let $c=au/2$. After translating $(0,a)$ to $(-a/2,a/2)$, the operator $\mc G_{u,a}$ is exactly the operator treated in \cite[Theorem~3 and Corollary~3]{KRD21}. If $\lambda_1(c)\ge\lambda_2(c)\ge\cdots\ge0$ are its eigenvalues, then \cite[Corollary~3]{KRD21} gives
\begin{align*}
    \lambda_k(c)
    &\ll \exp\Big(-\Omega\Big(\frac{k-2c/\pi}{\log(c+e)}\Big)\Big)
    && (k\ge2c/\pi),\\
    1-\lambda_k(c)
    &\ll \exp\Big(-\Omega\Big(\frac{2c/\pi-k}{\log(c+e)}\Big)\Big)
    && (k\le2c/\pi).
\end{align*}
Hence $\lambda_k(c)<\tau$ whenever
\[
    k\ge2c/\pi+O\big(\log(c+e)\log(1/\tau)\big),
\]
while $\lambda_k(c)>\tau$ whenever
\[
    k\le2c/\pi-O\big(\log(c+e)\big).
\]
Since $\tau<1/4$, the latter error term is also $O(\log(c+e)\log(1/\tau))$. Therefore
\[
    \Lambda(\tau,\mc G_{u,a})
    =\frac{2c}{\pi}+O\big(\log(c+e)\log(1/\tau)\big),
\]
and since $c=au/2$, the claim follows.
\end{proof}

For $a>0$, let $\mc K_a\colon L^2(0,a)\to L^2(0,a)$ be the compact truncated convolution operator with kernel $K$,
\begin{equation} \label{eq:A.5}
    (\mc K_a f)(x)=\frac12\int_0^a \operatorname{sech}\!\Big(\frac{x-y}{2}\Big)f(y)\,dy.
\end{equation}

\begin{lemma} \label{lem:A.3}
For $a>0$ and $\tau\in(0,1/4)$,
\[
    \Lambda(\tau,\mc K_a)=\frac{a}{\pi^2}\log(1/\tau)
    +O\Big(a+\log(1/\tau)\log(a\log(1/\tau)+e)\Big).
\]
\end{lemma}
\begin{proof}
If $F=f\one_{(0,a)}$, then Plancherel, together with \cref{eq:A.4} and the identity $\widehat K(\xi)=\pi\operatorname{sech}(\pi\xi)$, gives
\[
    \langle \mc K_a f,f\rangle
    =\frac{1}{2\pi}\int_{\R} \pi\operatorname{sech}(\pi\xi)|\widehat F(\xi)|^2\,d\xi,
    \qquad
    \langle \mc G_{u,a}f,f\rangle
    =\frac{1}{2\pi}\int_{-u}^u |\widehat F(\xi)|^2\,d\xi.
\]
Since
\[
    \pi\operatorname{sech}(\pi u)\one_{[-u,u]}(\xi)
    \le \pi\operatorname{sech}(\pi\xi)
    \le \pi\one_{[-u,u]}(\xi)+\pi\operatorname{sech}(\pi u),
\]
we obtain
\[
    \pi\operatorname{sech}(\pi u)\mc G_{u,a}
    \le \mc K_a
    \le \pi\mc G_{u,a}+\pi\operatorname{sech}(\pi u)\Id.
\]
Choose $u_\tau^{\pm}>0$ so that
\[
    \pi\operatorname{sech}(\pi u_\tau^+)=\tau/2,
    \qquad
    \pi\operatorname{sech}(\pi u_\tau^-)=8\tau.
\]
Then $u_\tau^{\pm}=\pi^{-1}\log(1/\tau)+O(1)$. Applying the min-max principle with $u=u_\tau^-$ and $u=u_\tau^+$, respectively, gives
\[
    \Lambda(1/8,\mc G_{u_\tau^-,a})
    \le \Lambda(\tau,\mc K_a)
    \le \Lambda(\tau/(2\pi),\mc G_{u_\tau^+,a}).
\]
Since
\[
    \frac{au_\tau^{\pm}}{\pi}=\frac{a}{\pi^2}\log(1/\tau)+O(a),
    \qquad
    \log(au_\tau^{\pm}+e)=O\big(\log(a\log(1/\tau)+e)\big),
\]
the claim follows from \cref{lem:A.2}.
\end{proof}

We can now estimate the eigenvalue-counting function of the covariance operator $T$.
\begin{proposition} \label{prop:A.4}
For all $\tau\in(0,1/4)$,
\[
    \Lambda(\tau,T)
    =\frac{1}{2\pi^2}\log^2(1/\tau)
    +O\big(\log^{3/2}(1/\tau)\sqrt{\log\log(1/\tau)}\big).
\]
\end{proposition}
\begin{proof}
The proof follows Laptev's argument \cite{Laptev74}. The quantitative input comes from \cref{lem:A.3}. By \cref{lem:A.1}, $T$ is unitarily equivalent to $\widetilde T$, so it suffices to compute $\Lambda(\tau,\widetilde T)$.

Fix integers $a,m\ge1$. Let $P_j\colon L^2(0,\infty)\to L^2((j-1)a,ja)$ be the orthogonal projection onto $[(j-1)a,ja)$, and let $Q_j\colon L^2(0,\infty)\to L^2(ja,\infty)$ be the orthogonal projection onto $[ja,\infty)$. Then
\[
    \widetilde T
    = \sum_{j=1}^m P_j\widetilde T P_j
      + \sum_{j=1}^m (Q_j\widetilde T P_j + P_j\widetilde T Q_j)
      + Q_m\widetilde T Q_m.
\]
The off-diagonal terms are not self-adjoint, so we estimate them via the singular-value counting function $N(\tau,\cdot)$. The Ky Fan inequality gives
\begin{align}
    \Lambda\big(\tau,\widetilde T\big)
    &\le \Lambda\Big(\frac{\tau}{2},\sum_{j=1}^m P_j\widetilde T P_j\Big)
      + \Lambda\Big(\frac{\tau}{4},Q_m\widetilde TQ_m\Big)
      + \sum_{j=1}^m N\Big(\frac{\tau}{8m},P_j\widetilde TQ_j\Big)
      + \sum_{j=1}^m N\Big(\frac{\tau}{8m},Q_j\widetilde TP_j\Big), \label{eq:A.6}\\
    \Lambda\Big(\frac{\tau}{2},\widetilde T\Big)
    &\ge \Lambda\Big(\tau,\sum_{j=1}^m P_j\widetilde T P_j\Big)
      - \Lambda\Big(\frac{\tau}{4},Q_m\widetilde TQ_m\Big)
      - \sum_{j=1}^m N\Big(\frac{\tau}{8m},P_j\widetilde TQ_j\Big)
      - \sum_{j=1}^m N\Big(\frac{\tau}{8m},Q_j\widetilde TP_j\Big). \label{eq:A.7}
\end{align}

We begin with the tail block $Q_m\widetilde TQ_m$. If $U_m\colon L^2(0,\infty)\to L^2(ma,\infty)$ is the translation unitary $(U_mf)(s)=f(s-ma)$, then
\[
    U_m^{-1}Q_m\widetilde TQ_mU_m = e^{-2ma}\widetilde T.
\]
Hence
\[
    \|Q_m\widetilde TQ_m\| = e^{-2ma}\|\widetilde T\| = e^{-2ma}\|T\| \le e^{-2ma}\Tr(T)=\frac12 e^{-2ma},
\]
where we used positivity of $T$. Choosing $m$ so that $e^{-2ma}<\tau/2$, all eigenvalues of $Q_m\widetilde TQ_m$ lie below $\tau/4$, and therefore
\[
    \Lambda\Big(\frac{\tau}{4},Q_m\widetilde TQ_m\Big)=0.
\]

Since $Q_j\widetilde T P_j=(P_j\widetilde TQ_j)^*$, the two off-diagonal terms have the same singular values. It therefore suffices to estimate $N(\tau/(8m),P_j\widetilde TQ_j)$. Let $U_j\colon L^2(ja,\infty)\to L^2(0,1)$ be the unitary map
\[
    (U_jf)(v)=(2v)^{-1/2}f\Big(ja-\frac12\log v\Big).
\]
Writing $x=ja-r$ with $0<r<a$, setting $u=e^{-2r}$, and changing variables $v=e^{-2(t-ja)}$, we obtain
\[
    (P_j\widetilde TQ_jU_j^{-1}g)(ja-r)
    = e^{-2ja}\int_0^1 \frac{\sqrt{2v}}{1+uv}g(v)\,dv.
\]
Since $uv\in[0,1]$, the function $z\mapsto(1+z)^{-1}$ is analytic on the disk $|z-1/2|<3/2$. Truncating its Taylor series about $z=1/2$ after $k$ terms yields a polynomial $p_k$ with
\[
    \sup_{0\le z\le1}|(1+z)^{-1}-p_k(z)|\ll3^{-k}.
\]
If $p_k(z)=\sum_{n=0}^k a_n z^n$, then the approximating kernel is
\[
    e^{-2ja}\sqrt{2v}\,p_k(uv)
    = e^{-2ja}\sum_{n=0}^k a_n u^n(\sqrt2\,v^{n+1/2}),
\]
so the approximating operator has rank at most $k+1$. The remainder kernel is uniformly $O(e^{-2ja}3^{-k})$ on a set of measure $a$, and therefore its Hilbert--Schmidt norm is $O(\sqrt a\,e^{-2ja}3^{-k})$. Hence
\[
    s_{k+2}(P_j\widetilde TQ_j)\ll \sqrt a\,e^{-2ja}3^{-k},
\]
and choosing
\[
    k\asymp \log_+\Big(\frac{8m\sqrt a\,e^{-2ja}}{\tau}\Big)
\]
gives
\begin{equation} \label{eq:A.8}
    N\Big(\frac{\tau}{8m},P_j\widetilde TQ_j\Big)
    \ll \log\Big(e+\frac{8m\sqrt a\,e^{-2ja}}{\tau}\Big).
\end{equation}

Finally, since the ranges of the $P_j$ are orthogonal,
\[
    \Lambda\Big(\tau,\sum_{j=1}^m P_j\widetilde T P_j\Big)
    = \sum_{j=1}^m \Lambda(\tau,P_j\widetilde T P_j).
\]
After translating $[(j-1)a,ja)$ to $[0,a)$ and dilating to $[0,2a)$, the block $P_j\widetilde T P_j$ is unitarily equivalent to
\[
    e^{-2(j-1)a}M\mc K_{2a}M,
    \qquad
    (Mf)(x)=e^{-x/2}f(x).
\]
Since the nonzero eigenvalues of $M\mc K_{2a}M$ and $\mc K_{2a}^{1/2}M^2\mc K_{2a}^{1/2}$ coincide, while $e^{-2a}\Id\le M^2\le\Id$, the min-max principle gives
\[
    \Lambda(\tau e^{2ja},\mc K_{2a})
    \le \Lambda(\tau,P_j\widetilde T P_j)
    \le \Lambda(\tau e^{2(j-1)a},\mc K_{2a}).
\]
Choose
\[
    m=\Big\lceil \frac{\log(2/\tau)}{2a}\Big\rceil+1.
\]
Note that $\Tr(\mc K_{2a})=a$, so only $O(1+1/a)$ of the parameters $\eta_j\in\{\tau e^{2ja},\tau e^{2(j-1)a}\}$ can exceed $1/4$, and for each such $j$ we have
\[
    \Lambda(\eta_j,\mc K_{2a})\le \eta_j^{-1}\Tr(\mc K_{2a})\le4a.
\]
Thus their total contribution is $O(a)$. Applying \cref{lem:A.3} to the remaining terms and summing the arithmetic progression
\[
    \log\big(1/(\tau e^{2ja})\big)=\log(1/\tau)-2ja
\]
yields
\begin{equation} \label{eq:A.9}
    \Lambda\Big(\tau,\sum_{j=1}^m P_j\widetilde T P_j\Big)
    = \frac{1}{2\pi^2}\log^2(1/\tau)
    + O\Big(a\log(1/\tau)+\frac{\log^2(1/\tau)}{a}\log(a\log(1/\tau)+e)\Big).
\end{equation}
Summing \cref{eq:A.8} over $1\le j\le m$ and using $m\asymp \log(1/\tau)/a$, we obtain
\[
    \sum_{j=1}^m N\Big(\frac{\tau}{8m},P_j\widetilde TQ_j\Big)
    \ll \frac{\log^2(1/\tau)}{a}\log(a\log(1/\tau)+e).
\]
Combining this with \cref{eq:A.9} and \cref{eq:A.6}, together with the fact that $\Lambda(\tau/4,Q_m\widetilde TQ_m)=0$, gives the upper bound
\[
    \Lambda(\tau,\widetilde T)
    \le \frac{1}{2\pi^2}\log^2(1/\tau)
    + O\Big(a\log(1/\tau)+\frac{\log^2(1/\tau)}{a}\log(a\log(1/\tau)+e)\Big).
\]
For the matching lower bound, apply \cref{eq:A.7} with $2\tau$ in place of $\tau$. Since replacing $\tau$ by $2\tau$ changes the main term by only $O(\log(1/\tau))$, which is absorbed by the displayed error term, we obtain
\[
    \Lambda(\tau,\widetilde T)
    \ge \frac{1}{2\pi^2}\log^2(1/\tau)
    - O\Big(a\log(1/\tau)+\frac{\log^2(1/\tau)}{a}\log(a\log(1/\tau)+e)\Big).
\]
Therefore
\[
    \Lambda(\tau,\widetilde T)
    = \frac{1}{2\pi^2}\log^2(1/\tau)
    + O\Big(a\log(1/\tau)+\frac{\log^2(1/\tau)}{a}\log(a\log(1/\tau)+e)\Big).
\]
Taking
\[
    a=\Big\lceil \sqrt{\log(1/\tau)\log\log(1/\tau)}\Big\rceil
\]
proves the claim.
\end{proof}

\subsection{Turning the counting bound into \texorpdfstring{\cref{lem:2.6}}{Lemma 2.6}}\label{subsec:A.2}
The counting estimate in \cref{prop:A.4} already yields the correct cubic scale for the $L^2$ small-ball probability. We record the resulting bound explicitly.

\begin{proposition} \label{prop:A.5}
Let $X$ be as in \cref{eq:2.2}. Then for all $\delta\in(0,1/4)$,
\[
    \log \PP\Big(\int_0^{\infty} e^{-t}X_t^2\,dt<\delta\Big)
    = -\frac{1}{12\pi^2}\log^3(1/\delta)
    + O\big(\log^{5/2}(1/\delta)\sqrt{\log\log(1/\delta)}\big).
\]
\end{proposition}
\begin{proof}
By \cref{eq:A.2}, it is enough to estimate
\[
    \log\PP(I<\delta),
    \qquad
    I=\sum_{k\ge1} \lambda_k\xi_k^2.
\]
Write
\[
    L(r)=\log\E e^{-rI}=-\frac12\sum_{k\ge1}\log(1+2r\lambda_k).
\]
It suffices to consider sufficiently small $\delta$. Since
\[
    \Lambda(\tau,T)=\sum_{k\ge1}\one_{\{\tau\le\lambda_k\}},
\]
exchanging summation and integration gives
\[
    -L(r)=r\int_0^{\lambda_1} \frac{\Lambda(\tau,T)}{1+2r\tau}\,d\tau.
\]
Also $\lambda_1\le\Tr(T)=1/2$, and $\tau\Lambda(\tau,T)\le\Tr(T)$, so $\Lambda(\tau,T)\le(2\tau)^{-1}\le2$ on $[1/4,\lambda_1]$. Therefore \cref{prop:A.4} and the change of variables $u=2r\tau$ imply
\[
    -L(r)=\frac{1}{12\pi^2}\log^3 r+O\big(\log^{5/2}r\sqrt{\log\log r}\big).
\]
For the upper bound, take
\[
    r_\delta=\delta^{-1}\log^2(1/\delta).
\]
Then $r_\delta\delta=\log^2(1/\delta)$ and $\log r_\delta=\log(1/\delta)+O(\log\log(1/\delta))$, so Markov's inequality yields
\[
    \PP(I<\delta)\le e^{L(r_\delta)+r_\delta\delta}
    = \exp\Big(-\frac{1}{12\pi^2}\log^3(1/\delta)
    + O\big(\log^{5/2}(1/\delta)\sqrt{\log\log(1/\delta)}\big)\Big).
\]
For the lower bound, take instead
\[
    r_\delta=\delta^{-1}\log^4(1/\delta).
\]
Then again $\log r_\delta=\log(1/\delta)+O(\log\log(1/\delta))$, and hence
\[
    -L(r_\delta)=\frac{1}{12\pi^2}\log^3(1/\delta)
    + O\big(\log^{5/2}(1/\delta)\sqrt{\log\log(1/\delta)}\big).
\]
Also $r_\delta\delta=\log^4(1/\delta)\gg -L(r_\delta)$, so
\[
    e^{-r_\delta\delta}\le\frac12 e^{L(r_\delta)}
\]
for all sufficiently small $\delta$. Since
\[
    e^{L(r_\delta)}=\E e^{-r_\delta I}\le \PP(I<\delta)+e^{-r_\delta\delta},
\]
this gives
\[
    \PP(I<\delta)\ge\frac12 e^{L(r_\delta)}.
\]
Taking logarithms yields
\[
    \log\PP(I<\delta)
    \ge L(r_\delta)+O(1)
    = -\frac{1}{12\pi^2}\log^3(1/\delta)
    + O\big(\log^{5/2}(1/\delta)\sqrt{\log\log(1/\delta)}\big).
\]
Combining the upper and lower bounds proves the proposition.
\end{proof}

\begin{proof}[Proof of \cref{lem:2.6}]
This is exactly \cref{prop:A.5}.
\end{proof}
 
\bibliographystyle{amsplain0}
\bibliography{main}

\end{document}